\documentclass[a4paper,12pt]{amsart}
\usepackage{amssymb}
\usepackage{ifthen}
 \usepackage[dvips]{graphicx}
\nonstopmode \numberwithin{equation}{section}
\usepackage{amssymb}
\usepackage{ifthen}
\usepackage{graphicx}
\usepackage{cite}
\usepackage{amsmath}
\usepackage[T1]{fontenc} 
\usepackage[utf8]{inputenc}
\usepackage[usenames,dvipsnames]{color}
\usepackage{color}
\usepackage[english]{babel}
\usepackage{fancyhdr}
\usepackage{fancybox}
\usepackage{tikz}

\nonstopmode \numberwithin{equation}{section}
\theoremstyle{plain}

\newtheorem*{lemA}{Lemma A}
\newtheorem*{lemB}{Lemma B}
\newtheorem*{lemC}{Lemma C}
\newtheorem*{lemD}{Lemma D}
\newtheorem*{lemE}{Lemma E}
\newtheorem*{lemF}{Lemma F}

\newtheorem{conj}{Conjecture}

\theoremstyle{definition}
\newtheorem{defn}{Definition}[section]

\newtheorem{thm}{Theorem}[section]

\newtheorem{prob}{Problem}[section]
\newtheorem{cor}{Corollary}[section]
\newtheorem{ques}{Question}[section]
\newtheorem{prop}{Proposition}[section]
\newtheorem{rem}{Remark}[section]
\newtheorem{lem}{Lemma}[section]

\newcounter{minutes}
\divide\time by 60
\newcounter{hours}
\multiply\time by 60
\addtocounter{minutes}{-\time}

\newcounter {own}
\def\theown {\thesection       .\arabic{own}}

\newenvironment{pf}[1][]{%
 \vskip 3mm
 \noindent
 \ifthenelse{\equal{#1}{}}%
  {{\slshape Proof. }}%
  {{\slshape #1.} }%
 }%
{\qed\bigskip}

\newcounter{alphabet}

\def\be{\begin{equation}}
\def\ee{\end{equation}}

\newcommand{\bee}{\begin{enumerate}}
\newcommand{\eee}{\end{enumerate}}

\newcommand{\blem}{\begin{lem}}
\newcommand{\elem}{\end{lem}}
\newcommand{\bthm}{\begin{thm}}
\newcommand{\ethm}{\end{thm}}
\newcommand{\bcor}{\begin{cor}}
\newcommand{\ecor}{\end{cor}}
\newcommand{\beg}{\begin{examp}}
\newcommand{\eeg}{\end{examp}}
\newcommand{\begs}{\begin{examples}}
\newcommand{\eegs}{\end{examples}}

\newcommand{\bdefn}{\begin{defn}}
\newcommand{\edefn}{\end{defn}}

\newcommand{\bprob}{\begin{prob}}
\newcommand{\eprob}{\end{prob}}
\newcommand{\bei}{\begin{itemize}}
\newcommand{\eei}{\end{itemize}}

\newcommand{\bcon}{\begin{conj}}
\newcommand{\econ}{\end{conj}}
\newcommand{\bcons}{\begin{conjs}}
\newcommand{\econs}{\end{conjs}}
\newcommand{\bprop}{\begin{prop}}
\newcommand{\eprop}{\end{prop}}
\newcommand{\br}{\begin{rem}}
\newcommand{\er}{\end{rem}}
\newcommand{\brs}{\begin{rems}}
\newcommand{\ers}{\end{rems}}
\newcommand{\bo}{\begin{obser}}
\newcommand{\eo}{\end{obser}}
\newcommand{\bos}{\begin{obsers}}
\newcommand{\eos}{\end{obsers}}
\newcommand{\bpf}{\begin{pf}}
\newcommand{\epf}{\end{pf}}
\newcommand{\ba}{\begin{array}}
\newcommand{\ea}{\end{array}}
\newcommand{\beq}{\begin{eqnarray}}
\newcommand{\beqq}{\begin{eqnarray*}}
\newcommand{\eeq}{\end{eqnarray}}
\newcommand{\eeqq}{\end{eqnarray*}}

\begin{document}

\title{Sharp Bohr radius involving Schwarz functions for certain classes of analytic functions}

\author{Molla Basir Ahamed}
\address{Molla Basir Ahamed, Department of Mathematics, Jadavpur University, Kolkata-700032, West Bengal, India.}
\email{mbahamed.math@jadavpuruniversity.in}

\author{Partha Pratim Roy}
\address{Partha Pratim Roy, Department of Mathematics, Jadavpur University, Kolkata-700032, West Bengal, India.}
\email{pproy.math.rs@jadavpuruniversity.in}

\subjclass[{AMS} Subject Classification:]{Primary 30C45, 30C80}

\def\thefootnote{}
\footnotetext{ {\tiny File:~\jobname.tex,
printed: \number\year-\number\month-\number\day,
          \thehours.\ifnum\theminutes<10{0}\fi\theminutes }
} \makeatletter\def\thefootnote{\@arabic\c@footnote}\makeatother

\begin{abstract}
The Bohr radius for an arbitrary class $\mathcal{F}$ of analytic functions of the form $f(z)=\sum_{n=0}^{\infty}a_nz^n$ on the unit disk $\mathbb{D}=\{z\in\mathbb{C} : |z|<1\}$ is the largest radius $R_{\mathcal{F}}$ such that every function $f\in\mathcal{F}$ satisfies the inequality 
\begin{align*}
	d\left(\sum_{n=0}^{\infty}|a_nz^n|, |f(0)|\right)=\sum_{n=1}^{\infty}|a_nz^n|\leq d(0, \partial f(\mathbb{D})),
\end{align*}
for all $|z|=r\leq R_{\mathcal{F}}$ , where $d(0, \partial f(\mathbb{D}))$ is the Euclidean distance. In this paper, our aim is to determine the sharp improved Bohr radius for the classes of analytic functions $f$ satisfying differential subordination relation $zf^{\prime}(z)/f(z)\prec h(z)$ and $f(z)+\beta zf^{\prime}(z)+\gamma z^2f^{\prime\prime}(z)\prec h(z)$,
where $h$ is the Janowski function. We show that improved Bohr radius can be obtained for Janowski functions as root of an equation involving Bessel function of first kind.  Analogues results are obtained in this paper for $\alpha$-convex functions and typically real functions, respectively. All obtained results in the paper are sharp and are improved version of [\textit{Bull. Malays. Math. Sci. Soc.} (2021) 44:1771-1785].
\end{abstract}

\maketitle
\pagestyle{myheadings}
\markboth{M. B. Ahamed and P. P. Roy}{Sharp improved Bohr radius for certain classes of analytic functions}
\tableofcontents
\section{\bf Introduction}
Let $\mathbb{D}:=\{z\in\mathbb{C} : |z|<1\}$ be the open unit disk in $\mathbb{C}$ and $\mathcal{H}(\mathbb{D}, \Omega)$ denote the class of analytic function mapping unit disk $\mathbb{D}$ into a domain $\Omega$. Let $\mathcal{A}$ denote the class of analytic functions in $\mathbb{D}$ normalized by $f(0)=0=f^{\prime}(0)-1$. Let $\mathcal{S}$ be the subclass of $\mathcal{S}$ consisting of univalent functions. For two analytic functions $f$ and $g$ in $\mathbb{D}$, the function $f$ is said to be subordinate to $g$, written as $f\prec g$, if there is an analytic function map $\omega : \mathbb{D}\to\mathbb{D}$ with $\omega(0)=0$ satisfying $f(z)=g(\omega(z))$. In particular, if the function $g$ is univalent in $\mathbb{D}$, then $f$ is subordinate to $g$ is equivalent to $f(0)=g(0)$ and $f(\mathbb{D})\subset g(\mathbb{D})$.\vspace{1.2mm}

Let $\varphi$ be the analytic function with positive real part in $\mathbb{D}$ that maps the unit disk $\mathbb{D}$ onto regions starlike with respect to $1$, symmetric with respect to the real axis and normalized by the conditions $\varphi(0)=1$ and $\varphi^{\prime}(0)>0$. For such functions, Ma and Minda (see \cite{Ma-Minda-1992})  introduced the following classes:
\begin{align*}
	\mathcal{ST}[\varphi]:=\bigg\{ f\in\mathcal{A} : \frac{z f^{\prime}(z)}{g(z)}\prec \varphi(z)\bigg\}
\end{align*}
 and 
 \begin{align*}
 	\mathcal{CV}[\varphi]:=\bigg\{ f\in\mathcal{A} : 1+\frac{zf^{\prime\prime}(z)}{f^{\prime}(z)}\prec \varphi(z)\bigg\}.
 \end{align*}
\begin{enumerate}
	\item[(a)] On taking $\varphi(z)=(1+Az)/(1+Bz)$, $(-1\leq B<A\leq 1)$, the class $\mathcal{ST}[\varphi]$ reduces to the familiar class consisting of Janowski starlike functions (see \cite{Janowski-APM-1973}), denoted by $\mathcal{ST}[A, B]$.\vspace{1.2mm}
	 
	\item[(b)] The special case $A=1-2\alpha$ and $B=-1$, that is, $\varphi(z)=(1+(1-2\alpha)z)/(1-z)$, $0\leq \alpha<1$ yields the classes $\mathcal{ST}[\alpha]$ and $\mathcal{ST}[\alpha]$ of starlike and convex functions of order $\alpha$, respectively. In particular, $\alpha=0$, that is, $A=1$ and $B=-1$ leads to the usual classes $\mathcal{ST}$ and $\mathcal{CV}$ of starlike and convex functions, respectively. \vspace{1.2mm}
	
	\item[(c)] For $A=1$ and $B=(1-M)/M$, $M>1/2$, we obtain the class $\mathcal{ST}[{M}]$, introduced by Janowski (see \cite{Janowski-APM-1970}).\vspace{1.2mm}
	
	\item[(d)] In $1939$, Robertson (see \cite{Robertson-AM-1936}) introduced a very well-known class $\mathcal{ST}^{(\beta)}:=\mathcal{ST}[-\beta, \beta]$, $0<\beta\leq 1$. \vspace{1.2mm}
	
	\item[(e)] Also, $\mathcal{ST}_{(\beta)}:=\mathcal{ST}[\beta, 0]$ leads to a class which was introduced by MacGregor in \cite{MacGregor-PAMS-1963}.
\end{enumerate}
In $1914$, Bohr (see \cite{Bohr-PLMS-1914}) discovered that if a power series of analytic function $f(z)=\sum_{n=0}^{\infty}a_nz^n$ converges in the unit disk $\mathbb{D}$ with $|f(z)|<1$ for $z\in\mathbb{D}$, then $M_f(r):=\sum_{n=0}^{\infty}|a_nz^n|\leq 1$ for $|z|=r\leq 1/6$. Weiner, Riesz and Schur independently proved that the Bohr's result holds in the disk $ \mathbb{D}_{\frac{1}{3}}:=\{z\in\mathbb{C} : |z|\leq \frac{1}{3}\}$ and the radius $1/3$ is best possible. For the class $\mathcal{B}:=\mathcal{H}(\mathbb{D}, \mathbb{D})$, the radius $1/3$ is known as the Bohr radius and the inequality $M_f(r)\leq 1$ is commonly known as the Bohr inequality. In recent years, there has been a growing interest determining the Bohr radius for various classes of functions in one and several complex variables (see \cite{Alkhaleefah-Kayumov-Ponnusamy-PAMS-2019,Kayumov-CRACAD-2018,Kayumov-Khammatova-JMAA-2021,Kayumov-MJM-2022}), class of harmonic mappings (see \cite{Ahamed-Allu-Halder-AMP-2021,Ahamed-Allu-RMJ-2022}), class of quasiconformal mappings (see \cite{Kayumov-Ponnusamy-Shakirov-Math.Nah.-2018,Liu-Ponnusamy-Wang-RACSAM-2020,Ahamed-Ahamed-MJM-2024}). \vspace{1.2mm}

Using the Euclidean distance $d$, the Bohr inequality for a function of the form $f(z)=\sum_{n=0}^{\infty}a_nz^n$ is written as 
\begin{align*}
	d\left(\sum_{n=0}^{\infty}|a_nz^n|, |a_0|\right)=\sum_{n=1}^{\infty}|a_nz^n|\leq 1-|a_0|=1-|f(0)|=d(f(0), \partial \mathbb{D}),
\end{align*}
where $\partial\mathbb{D}$ is the boundary of the disc $\mathbb{D}$.\vspace{1.2mm}

For any domain $\Omega$ and all functions $f\in\mathcal{H}(\mathbb{D}, \Omega)$, the Bohr radius is the largest radius $R_{\mathcal{H}}>0$ such that 
\begin{align*}
	d\left(\sum_{n=0}^{\infty}|a_nz^n|, |f(0)|\right)=\sum_{n=1}^{\infty}|a_nz^n|\leq 1-|f(0)|=d(f(0), \partial \mathbb{D}),
\end{align*}
for $|z|=r\leq R_{\mathcal{H}}$.\vspace{1.2mm}

Before we continue the discussion, we fix some notations here. Throughout the discussion, we assume that $\mathcal{B}=\{f\in H_\infty:||f||_\infty\leq 1\}$ and for $m\in\mathbb{N}=\{1,2,\ldots\}$, we define
\begin{align*}
\mathcal{B}_m=\{\omega\in\mathcal{B}:\omega(0)=\cdots=\omega^{(m-1)}(0)=0\;\mbox{and}\;\omega^{(m)}(0)\neq 0\}.
\end{align*} 
By using the Schwarz function in place of the initial power series coefficients, many authors  established  Bohr-type inequalities for bounded analytic functions in the unit disk (see \cite{Gang Liu-JMAA-2021,Huang-Liu-Ponnusamy-AMP-2020,Allu-Arora-JMMA-2022}). In this paper, we continue the research by determining the sharp Bohr radius using Schwarz functions in the Bohr inequalities for some classes of analytic functions on the unit disk $\mathbb{D}$. Our main results improve various existing findings in the literature, as discussed in the subsequent sections. 
\section{\bf Bohr inequality for Janowski starlike functions}
In this section, we will derive an improved Bohr radius for a specific class of Janowski starlike functions. First, we will discuss some preliminary results before stating the main findings. In \cite[Theorem 3, p.738]{Aouf-IJMMS-1987}, it was shown that the following lemmas for Janowski starlike functions are crucial in proving the main results when $p=1$ and $\alpha=0$. 
\begin{lemA}\emph{(see \cite[Lemma 1]{Anand-Jain-Kumar-BMMSS-2021})}
	If $f(z)=z+\sum_{n=2}^{\infty}a_nz^n\in\mathcal{ST}[A, B]$, then 
	\begin{align}\label{Eq-2.1}
		|a_n|\leq \prod_{k=0}^{n-2}\frac{|(B-A)+Bk|}{k+1}\; (n\geq 2)
	\end{align}
	and these bounds are sharp.
\end{lemA}
\begin{lemB}\emph{(see \cite[Theorem 4, p.315]{Janowski-APM-1973})}
	If $f(z)=z+\sum_{n=2}^{\infty}a_nz^n\in\mathcal{ST}[A, B]$, then for $|z|=r$; $(0\leq r<1)$,
	\begin{align}\label{Eq-2.2}
		l_{(-A, -B)}(r)\leq |f(re^{i\theta})|\leq l_{(A, B)}(r),
	\end{align}
	where $l_{(A, B)} : \mathbb{D}\to\mathbb{C}$ is given by 
	\begin{align}\label{Eq-2.3}
		l_{(A, B)}(z):=\begin{cases}
			z\left(1+Bz\right)^{(A-B)/B},\; B\neq 0;\vspace{2mm}\\
			ze^{Az},\;\;\;\;\;\;\;\;\;\;\;\;\;\;\;\;\;\;\;\;\;\; B=0.
		\end{cases}
	\end{align}
	The result is sharp.
\end{lemB}
\subsection{\bf Refined Bohr inequalities for Janowski starlike functions}
The Bohr radius for the Janowski starlike functions is obtained recently in \cite[Theorem 1]{Anand-Jain-Kumar-BMMSS-2021} and it is shown that the radius is sharp. In this section, we continue the study to establish a refined version of the Bohr inequality \cite[Theorem 1]{Anand-Jain-Kumar-BMMSS-2021} aiming to show the corresponding radius is sharp. To achieve our goal, Bessel function of the first kind has a pivotal role. Before, we state our main result, we recall here modified Bessel function of the first kind.
\subsection{\bf Modified Bessel function of the first kind}
Bessel functions, first defined by the mathematician Daniel Bernoulli and then generalized by Friedrich Bessel, are canonical solutions $y(x)$ of Bessel's differential equation
\begin{align*}
	\displaystyle x^{2}{\frac {d^{2}y}{dx^{2}}}+x{\frac {dy}{dx}}+\left(x^{2}-\alpha ^{2}\right)y=0
\end{align*}
for an arbitrary complex number 
$\alpha$, which represents the order of the Bessel function. Although, $\alpha$  and $-\alpha$ produce the same differential equation, it is conventional to define different Bessel functions for these two values in such a way that the Bessel functions are mostly smooth functions of 
$\alpha$. Bessel's equation arises when finding separable solutions to Laplace's equation and the Helmholtz equation in cylindrical or spherical coordinates. Bessel functions are therefore especially important for many problems of wave propagation and static potentials. In solving problems in cylindrical coordinate systems, one obtains Bessel functions of integer order $\alpha=n$; in spherical problems, one obtains half-integer orders $\alpha=n+\frac{1}{2}$. For example: Electromagnetic waves in a cylindrical waveguide,
Pressure amplitudes of inviscid rotational flows,
Heat conduction in a cylindrical object,
Modes of vibration of a thin circular or annular acoustic membrane (such as a drumhead or other membranophone) or thicker plates such as sheet metal (see Kirchhoff–Love plate theory, Mindlin–Reissner plate theory),
Diffusion problems on a lattice,
Solutions to the radial Schrödinger equation (in spherical and cylindrical coordinates) for a free particle
Position space representation of the Feynman propagator in quantum field theory. \vspace{1.2mm}

Bessel functions of the first kind, denoted as $\mathcal{J}_{\alpha}(x)$, are solutions of Bessel's differential equation. For integer or positive $\alpha$, Bessel functions of the first kind are finite at the origin $(x=0)$; while for negative non-integer $\alpha$, Bessel functions of the first kind diverge as $x$ approaches zero. It is possible to define the function by $\alpha$ times a Maclaurin series (note that $\alpha$ need not be an integer, and non-integer powers are not permitted in a Taylor series), which can be found by applying the Frobenius method to Bessel's equation:
\begin{align*}
	\mathcal{J}_{\alpha}(x)=\sum_{n=0}^{\infty}\frac{(-1)^n}{n! \Gamma(n+\alpha+1)}\left(\frac{x}{2}\right)^{2n+\alpha},
\end{align*}
where $\Gamma (n)$ is the gamma function, a shifted generalization of the factorial function to non-integer values. The Bessel function of the first kind is an entire function if $\alpha$ is an integer, otherwise it is a multi-valued function with singularity at zero.\vspace{1.2mm}

The Bessel functions can be expressed in terms of the generalized hypergeometric series as 
\begin{align*}
	\mathcal{J}_{\alpha}(x)=\frac{\left(\frac{x}{2}\right)^{\alpha}}{\Gamma(\alpha+1)}{}_0 F_1\left(\alpha+1;\; -\frac{x^2}{4}\right).
\end{align*}
Another definition of the Bessel function, for integer values of $n$, is possible using an integral representation:
\begin{align*}
	\mathcal{J}_{\alpha}(x)=\frac{1}{\pi}\int_{0}^{\pi}\cos\left(n\theta-x\sin\theta\right)d\theta=\frac{1}{2\pi}\int_{-\pi}^{\pi}e^{i(n\theta-x\sin\theta)}d\theta,
\end{align*}
which is also called Hansen-Bessel formula.\vspace{1.2mm}

In particular, if $\alpha=0$, then it is easy to see that
\begin{align}\label{Eq-2.6}
	\mathcal{J}_0(x)=\sum_{n=0}^{\infty}\frac{1}{n! \Gamma(n+1)}\left(\frac{x}{2}\right)^{2n}=\sum_{n=0}^{\infty}\frac{1}{4(n!)^2}x^{2n}.
\end{align}
Further, putting $x=2ar$, it follows from \eqref{Eq-2.6} that
\begin{align*}
	\mathcal{J}_0(2ar)=\sum_{n=0}^{\infty}\frac{(a^2r^2)^{n}}{(n!)^2}.
\end{align*}
Our first result regarding the refined Bohr inequality for the class $\mathcal{ST}[A, B]$ is as follows: we demonstrate that the radius can be determined by utilizing the Bessel function $\mathcal{J}_0(x)$ in the Bohr inequality involving the Schwarz function.
\begin{thm}\label{Th-2.2}
	Let $f(z)=z+\sum_{n=2}^{\infty}a_nz^n$ be in the class $\mathcal{ST}[A, B]$ and $\Phi : [0, 1]\to [0, \infty)$ be a continuous increasing function. Then 
	\begin{align*}
		|\omega(z)|+\sum_{n=2}^{\infty}|a_n||\omega(z)|^n+\Phi(|\omega(z)|)\sum_{n=2}^{\infty}|a_n|^2|\omega(z)|^{2n}\leq d(0, \partial f(\mathbb{D}))
	\end{align*}
	for $|z|=r\leq R^{1,{\Phi(|\omega(z)|)}}_{A, B}$, where $R^{1,{\Phi(|\omega(z)|)}}_{A, B}$ is the root in $(0, 1)$ of the equations $	\Psi_1(r)=0$ and $	\Psi_2(r)=0$, 
	\begin{align*}
		\Psi_1(r):&=r+\sum_{n=2}^{\infty}\left(\prod_{k=0}^{n-2}\frac{|(B-A)+Bk|}{k+1}\right)r^n+\Phi(r)\sum_{n=2}^{\infty}\left(\prod_{k=0}^{n-2}\frac{|(B-A)+Bk|}{k+1}\right)^2r^{2n}\\&\quad-(1-B)^{(A-B)/B},\; B\neq 0
	\end{align*}
	and 
	\begin{align}\label{Eq-2.5}
		\Psi_2(r):=re^{Ar}+\Phi(r)r^2\left(\mathcal{J}_0(2Ar)-1\right)-e^{-A},\; B=0.
	\end{align}
	Each radius $R^{1,{\Phi(|\omega(z)|)}}_{A, B}$ is sharp. 
\end{thm}
\begin{table}[ht]
\centering
\begin{tabular}{|l|l|l|l|l|l|l|}
		\hline
		$A$& $0.1 $&$0.3 $& $0.5$& $0.8$& $0.9$&$1 $\\
		\hline
		$R^{1,\omega(z)}_{A, B}$& $0.829$&$0.610 $& $0.473$& $0.340$& $0.306 $&$0.277$\\
		\hline
		$ R^{1,e^\omega(z)}_{A, B}$& $0.823$&$0.0.6 $& $0.465$& $0.334$& $0.302 $&$0.273 $\\
		\hline
		$ R^{1,\sin(\omega(z))}_{A, B}$& $0.829$&$0.610 $& $0.473$& $0.340$& $0.306 $&$0.277 $ \\
		\hline
		$R^{1,\log(1+\omega(z))}_{A, B}$& $0.830$&$0.611 $& $0.474$& $0.341$& $0.306 $&$0.277$ \\
		\hline
		$ R^{1,\frac{1}{2}+\frac{\omega(z)}{1-\omega(z)}}_{A, B}$& $0.814$&$0.599 $&  $0.467$& $0.336$& $0.304$& $0.275$\\
		\hline
\end{tabular}\vspace{2.5mm}
	\caption{The table exhibits the roots $R^{1,\Phi(|\omega(z)|)}_{A, B}$ of the equation \eqref{Eq-2.5} for different choice of $\Phi(\omega(z))=\omega(z)$, $ e^{\omega(z)} $, $\sin(\omega(z))$, $\log(1+\omega(z))$ and $\frac{1}{2}+\frac{\omega(z)}{1-\omega(z)}$ with $\omega(z)=z$ for different values of $A$.}
\end{table}
In particular, if $\Phi(|\omega(z)|)=0$, then we get the following corollary.
\begin{cor}\label{Cor-2.5}
	Let $f(z)=z+\sum_{n=2}^{\infty}a_nz^n$ be in the class $\mathcal{ST}[A, B]$. Then 
	\begin{align*}
		|\omega(z)|+\sum_{n=2}^{\infty}|a_n||\omega(z)|^n\leq d(0, \partial f(\mathbb{D}))
	\end{align*}
	for $|z|=r\leq R^{1,0}_{A, B}$, where $R^{1,0}_{A, B}$ is the root in $(0, 1)$ of the equations 
	\begin{align*}
		r+\sum_{n=2}^{\infty}\left(\prod_{k=0}^{n-2}\frac{|(B-A)+Bk|}{k+1}\right)r^n-(1-B)^{(A-B)/B}=0,\; \emph{when}\; B\neq 0
	\end{align*}
	and 
	\begin{align*}
		re^{Ar}-e^{-A}=0,\; \emph{when}\; B=0.
	\end{align*}
	Each radius $R^{1,0}_{A, B}$ is sharp. 
\end{cor}
\begin{rem}
	Corollary \ref{Cor-2.5} improves the result \cite[Theorem 1]{Anand-Jain-Kumar-BMMSS-2021} involving Schwartz functions. The radius $R^{1,0}_{A, B}$ is exactly the same as the radius $r^*$ mentioned in \cite[Theorem 1]{Anand-Jain-Kumar-BMMSS-2021}. Moreover, when $\Phi(|\omega(z)|)\neq0$, we obtain an improvement.
\end{rem}
\begin{rem}
 In particular, if $\omega(z)=z$, then Theorem \ref{Th-2.2} is an improvement over the earlier result from \cite[Theorem 1]{Anand-Jain-Kumar-BMMSS-2021} in perspective of a general version of positive refinement term $\Phi(z)$.
\end{rem}
\subsection{\bf Improved Bohr inequalities for Janowski starlike functions}
Previous and recent studies into Bohr inequalities for analytic functions, where the term $|f(z)|^p$ was examined with $p$ values of $1$ and $2$, have been expanded by Ponnusamy \emph{et al.} \cite{Alkhaleefah-Kayumov-Ponnusamy-PAMS-2019}, the authors established that $p$ can be in the range $(0,2]$. However, for harmonic functions, $p$ can be any natural number as discussed in \cite{Ahamed-CVEE-2022}.\vspace{1.2mm} 

In view of this observation, it is natural to raise the following question.
\begin{ques}\label{Q-1}
	Can we improve the Bohr inequality \cite[Theorem 1]{Anand-Jain-Kumar-BMMSS-2021} for the class of Janowski starlike functions $\mathcal{S}[A,B]$ by including the term $|f(z)|^p$ for all $p \geq 0$?
\end{ques}
Our next result is answering the Question \ref{Q-1}  for improved Bohr inequalities for the class $\mathcal{ST}[A,B]$ including the term $|f(z)|^p$ for $p\geq0$  and involving Schwartz function.
\begin{thm}\label{Th-2.1}
	Let $f(z)=z+\sum_{n=2}^{\infty}a_nz^n$ be in the class $\mathcal{ST}[A, B]$. Then, for $p \geq0$ 
	\begin{align*}
		|f(re^{i\theta})|^p+|\omega(z)|+\sum_{n=2}^{\infty}|a_n(\omega(z))^n|\leq d(0, \partial f(\mathbb{D}))
	\end{align*}
	for $|z|=r\leq R^p_{A, B}$, where $R^p_{A, B}\in (0, 1)$ is the root of the equations $\Psi_3(r)=0$ and $\Psi_4(r)=0$, where
	\begin{align*}
		\Psi_3(r):=\displaystyle&\left(r(1+Br)^{(A-B)/B}\right)^p\\&+r+\sum_{n=2}^{\infty}\left(\prod_{k=0}^{n-2}\frac{|(B-A)+Bk|}{k+1}\right)r^n-(1-B)^{(A-B)/B},\; B\neq 0
	\end{align*}
	and 
	\begin{align}\label{Eqn-2.4}
		\Psi_4(r):= r^pe^{Apr}+re^{Ar}-e^{-A},\; B=0.
	\end{align}
	Each radius $R^p_{A, B}$ is sharp.
\end{thm}
\begin{rem}
	In particular when $\omega(z)=z$, then we see that Theorem \ref{Th-2.1} improves the result \cite[Theorem 1]{Anand-Jain-Kumar-BMMSS-2021}.
\end{rem}
\begin{table}[ht]
	\centering
	\begin{tabular}{|l|l|l|l|l|l|l|l|l|}
		\hline
		$p$& $R^p_{0.1} $&$R^p_{0.3} $& $R^p_{0.4}$& $R^p_{0.5}$& $R^p_{0.7}$&$R^p_{0.8}  $&$R^p_{0.9}$ &$R^p_{1}$ \\
		\hline
		$1$& $0.433$&$0.334 $& $0.297$& $0.265$& $0.213 $&$0.192$& $0.173$& $0.157$\\
		\hline
		$ 3$& $0.603$&$0.486 $& $0.440$& $0.400$& $0.333 $&$0.305 $& $0.279$& $0.256$\\
		\hline
		$ 7$& $0.707$&$0.569 $& $0.513$& $0.462$& $0.376 $&$0.339 $& $0.307$& $0.277$\\
		\hline
		$10$& $0.741$&$0.593 $& $0.530$& $0.473$& $0.379 $&$0.341 $& $0.308$& $0.278$\\
		\hline
		$ 25$& $0.804$&$0.615 $& $0.540$& $0.477$& $0.380 $&$0.420 $& $0.309$& $0.279$\\
		\hline
	\end{tabular}\vspace{2.5mm}
	\caption{The table exhibits the roots $R^p_{A,0}=R^p_A$(say) of the equation \eqref{Eqn-2.4} for different values of $A(0<A\leq1)$ and $p$.}
	\label{T-6}
\end{table}
\subsection{\bf Bohr inequality involving $S_r/\pi$ for the class $\mathcal{ST}[A, B]$}
 Let $ f $ be holomorphic in $ \mathbb{D} $, and for $ 0<r<1 $, let $ \mathbb{D}_r=\{z\in\mathbb{C} : |z|<r\} $. Throughout the paper, $ S_r:=S_r(f) $ denoted the planar integral 
 \begin{align*}
 	S_r:=\int_{\mathbb{D}_r}|f^{\prime}(z)|^2dA(z).
 \end{align*}
 Note that if $ f(z)=\sum_{n=0}^{\infty}a_nz^n $, then $ S_r:=\pi\sum_{n=1}^{\infty}n|a_n|^2r^{2n} $ (see \cite{Duren-1983}). It is well-known that if $ f $ is a univalent function, then  $ S_r $ is the area of the image of the sub-disk $ \mathbb{D}_r:=\{z\in\mathbb{D} : |z|<r\} $ under the mapping $ f $ (see \cite{Duren-1983}). The quantity $ S_r $ is non-negative and has been used extensively to study the improved versions of Bohr inequality and Bohr radius for the class of analytic functions (see e.g. \cite{Aha-Aha-CMFT-2023,Ismagilov-2020-JMAA,Kayumov-CRACAD-2018,Kayumov-Khammatova-JMAA-2021}) as well as harmonic mappings (see e.g. \cite{Ahamed-Allu-Halder-AMP-2021,Ahamed-CVEE-2021}).
Since 
\begin{align*}
	\mathcal{J}_0(x)=\sum_{n=0}^{\infty}\frac{\left(\frac{x}{2}\right)^{2n}}{(n!)^2}\; \mbox{and}\; \mathcal{J}^{\prime}_0(x)=\frac{1}{x}\sum_{n=0}^{\infty}n\frac{\left(\frac{x}{2}\right)^{2n}}{(n!)^2}
\end{align*}
it follows that 
\begin{align*}
	\mathcal{J}_0(2Ar)=\sum_{m=0}^{\infty}\frac{A^{2m}r^{2m}}{(m!)^2}\; \mbox{and}\; Ar\mathcal{J}^{\prime}_0(2Ar)=\sum_{m=0}^{\infty}m\frac{A^{2m}r^{2m}}{(m!)^2}.
\end{align*}
Hence,
\begin{align*}
	\sum_{n=2}^{\infty}n\left(\prod_{k=0}^{n-2}\frac{A}{k+1}\right)^2r^{2n}&=\sum_{n=2}^{\infty}n\left(\frac{A^{n-1}}{(n-1)!}\right)^2r^{2n}\\&=\sum_{m=1}^{\infty}(m+1)\left(\frac{A^{m}}{m!}\right)^2r^{2m+2}\\&=r^2\left(\sum_{m=1}^{\infty}\frac{A^{2m}r^{2m}}{(m!)^2}+\sum_{m=1}^{\infty}m\frac{A^{2m}r^{2m}}{(m!)^2}\right)\\&=r^2\left(-1+\mathcal{J}_0(2Ar)+Ar\mathcal{J}^{\prime}_0(2Ar)\right).
\end{align*}
For $f\in \mathcal{ST}[A, B]$, a simple computation shows that
\begin{align}\label{Eq-6.1}
	\frac{S_r}{\pi}&=\sum_{n=1}^{\infty}n|a_n|^2r^{2n}\\&\leq\nonumber\begin{cases}
		\displaystyle r^2+\sum_{n=2}^{\infty}n\left(\prod_{k=0}^{n-2}\frac{|(B-A)+Bk|}{k+1}\right)^2r^{2n},\; \mbox{when}\; B\neq 0\vspace{2mm}\\
		\displaystyle r^2+\sum_{n=2}^{\infty}n\left(\prod_{k=0}^{n-2}\frac{A}{k+1}\right)^2r^{2n},\;\;\;\;\;\;\;\;\;\;\;\;\;\;\;\;\;  \mbox{when}\; B=0
	\end{cases}
	\\&=\nonumber\begin{cases}
		\displaystyle r^2+\sum_{n=2}^{\infty}n\left(\prod_{k=0}^{n-2}\frac{|(B-A)+Bk|}{k+1}\right)^2r^{2n},\; \mbox{when}\; B\neq 0\vspace{2mm}\\
		\left(\mathcal{J}_0(2Ar)+Ar\mathcal{J}^{\prime}_0(2Ar)\right)r^2,\;\;\;\;\;\;\;\;\;\;\;\;  \mbox{when}\; B=0.
	\end{cases}
\end{align}
Based on the discussions above, we now present our next result, which pertains to the term $S_r/\pi$ and aims to improve the Bohr inequality with Schwartz functions.
\begin{thm}\label{Th-6.1}
	Let $f(z)=z+\sum_{n=2}^{\infty}a_nz^n$ be in the class $\mathcal{ST}[A, B]$ and $\Phi : [0, 1]\to [0, \infty)$ be a continuous increasing function. Then,  
	\begin{align*}
		|\omega(z)|+\sum_{n=2}^{\infty}|a_n(\omega(z))^n|+\Phi(|\omega(z)|)\left(\frac{S_r}{\pi}\right)\leq d(0, \partial f(\mathbb{D}))
	\end{align*}
	for $|z|=r\leq R^{2,{\Phi(|\omega(z)|)}}_{A, B}$, where $R_{(A, B)}\in (0, 1)$ is the root of the equations $G_1(r)=0$ and $G_2(r)=0$, 
	\begin{align*}
		G_1(r):=\displaystyle&r+\sum_{n=2}^{\infty}\left(\prod_{k=0}^{n-2}\frac{|(B-A)+Bk|}{k+1}\right)r^n+\Phi(r)\sum_{n=1}^{\infty}n\left(\prod_{k=0}^{n-2}\frac{|(B-A)+Bk|}{k+1}\right)^2r^{2n}\\&\quad-(1-B)^{(A-B)/B},\; B\neq 0
	\end{align*}
	and 
	\begin{align}\label{Eqn-2.8}
		G_2(r):=re^{Ar}+\Phi(r)\left(\mathcal{J}_0(2Ar)+Ar\mathcal{J}^{\prime}_0(2Ar)\right)r^2-e^{-A},\; B=0.
	\end{align}
	Each radius $R^{2,{\Phi(|\omega(z)|)}}_{A, B}$ is sharp.
\end{thm}
\begin{table}[ht]
	\centering
	\begin{tabular}{|l|l|l|l|l|l|l|}
		\hline
		$A$& $0.1 $&$0.3 $& $0.5$& $0.8$& $0.9$&$1 $\\
		\hline
		$R^{2,\omega(z)}_{A, B}$& $0.622$&$0.513 $& $0.423$& $0.319$& $0.291 $&$0.265$\\
		\hline
		$ R^{2,e^\omega(z)}_{A, B}$& $0.488$&$0.415 $& $0.351$& $0.273$& $0.252 $&$0.231 $\\
		\hline
		$ R^{2,\sin(\omega(z))}_{A, B}$& $0.629$&$0.516 $& $0.424$& $0.319$& $0.291 $&$0.265$ \\
		\hline
		$R^{2,\log(1+\omega(z))}_{A, B}$& $0.648$&$0.527 $& $0.430$& $0.321$& $0.292 $&$0.266$ \\
		\hline
		$ R^{2,\left(\frac{1}{2}+\frac{\omega(z)}{1-\omega(z)}\right)}_{A, B}$& $0.440$&$0.380 $&  $0.326$& $0.257$& $0.237$& $0.219$\\
		\hline
	\end{tabular}\vspace{2.5mm}
	\caption{The table exhibits the roots $R^{2,\Phi(|\omega(z)|)}_{A, B}$ of the equation \eqref{Eqn-2.8} for different choice of $\Phi(\omega(z))=\omega(z)$, $ e^{\omega(z)} $, $\sin(\omega(z))$, $\log(1+\omega(z))$ and $\frac{1}{2}+\frac{\omega(z)}{1-\omega(z)}$ with $\omega(z)=z$ for different values of $A$.}
\end{table}
\begin{rem}
	In particular when $\Phi(|\omega(z)|)=0$, we get a general version of \cite[Theorem 1]{Anand-Jain-Kumar-BMMSS-2021} involving Schwartz function and when $\Phi(|\omega(z)|)\neq0$, then we get an improved result.
\end{rem}
\begin{rem}
	In particular, if $\omega(z)=z$, then the Theorem \ref{Th-6.1} improves the result \cite[Theorem 1]{Anand-Jain-Kumar-BMMSS-2021}  involving a general positive function $\Phi(r)$.
\end{rem}
\begin{proof}[\bf Proof of the Theorem \ref{Th-2.2}]
		Since $f(z)=z+\sum_{n=2}^{\infty}a_nz^n\in\mathcal{ST}[A, B]$, using the inequality \eqref{Eq-2.2}, we have \begin{align}\label{Eq-2.8}
		d(0, \partial f(\mathbb{D}))=\inf_{z\in\partial f(\mathbb{D})}|f(z)-f(0)|=\inf_{z\in\partial f(\mathbb{D})}|f(z)|\geq l_{(-A, -B)}(1).
	\end{align}
	For $A>0$, in view of \eqref{Eq-2.6}, an easy computation gives that
	\begin{align}\label{Eq-2.7}
		\sum_{n=2}^{\infty}\left(\prod_{k=0}^{n-2}\frac{A}{k+1}\right)^2r^{2n}&=\sum_{n=2}^{\infty}\left(\frac{A^{n-1}}{(n-1)!}\right)^2r^{2n}\\&\nonumber=r^2\sum_{m=1}^{\infty}\left(\frac{A^{m}r^m}{m!}\right)^2=r^2\left(I_0(2Ar)-1\right).
	\end{align}
	Since $\omega : \mathbb{D}\to\mathbb{D}$ be such that $|\omega(z)|\leq |z|=r$, as the function $ \Phi$ is increasing, we have $\Phi(|\omega(z)|)\leq \Phi(r)$. Hence, by the Lemma A and Lemma B, and using \eqref{Eq-2.7}, we have
	\begin{align*}
		|\omega(z)|&+\sum_{n=2}^{\infty}|a_n||\omega(z)|^n+\Phi(|\omega(z)|)\sum_{n=2}^{\infty}|a_n|^2|\omega(z)|^{2n}\\&\leq r+\sum_{n=2}^{\infty}\left(\prod_{k=0}^{n-2}\frac{|(B-A)+Bk|}{k+1}\right)r^n\\&\quad+\Phi(r)\sum_{n=2}^{\infty}\left(\prod_{k=0}^{n-2}\frac{|(B-A)+Bk|}{k+1}\right)^2r^{2n}\\&\leq \begin{cases}
			\displaystyle r+\sum_{n=2}^{\infty}\left(\prod_{k=0}^{n-2}\frac{|(B-A)+Bk|}{k+1}\right)r^n+\Phi(r)\sum_{n=2}^{\infty}\left(\prod_{k=0}^{n-2}\frac{|(B-A)+Bk|}{k+1}\right)^2r^{2n},\; B\neq 0\vspace{2mm}\\
			\displaystyle re^{Ar}+\Phi(r)\sum_{n=2}^{\infty}\left(\prod_{k=0}^{n-2}\frac{A}{k+1}\right)^2r^{2n},\; B=0
		\end{cases}\\&=\begin{cases}
			\displaystyle r+\sum_{n=2}^{\infty}\left(\prod_{k=0}^{n-2}\frac{|(B-A)+Bk|}{k+1}\right)r^n+\Phi(r)\sum_{n=2}^{\infty}\left(\prod_{k=0}^{n-2}\frac{|(B-A)+Bk|}{k+1}\right)^2r^{2n},\; B\neq 0\vspace{2mm}\\
			\displaystyle re^{Ar}+r^2\Phi(r)\left(\mathcal{J}_0(2Ar)-1\right),\; B=0
		\end{cases}
		\\&\leq l_{-A, -B}(1)\\&\leq d(0, \partial f(\mathbb{D}))
	\end{align*}
	for $|z|=r\leq R^{1,{\Phi(|\omega(z)|)}}_{A, B}$, where $R^{1,{\Phi(|\omega(z)|)}}_{A, B}$ as in the statement of the theorem. Thus the inequality is established.\vspace{1.2mm}
	
	In order to show that the radius $ R^{1,{\Phi(|\omega(z)|)}}_{A, B}$, we consider the function $f=l_{(A, B)}$ which is given by \eqref{Eq-2.3}. In view of \eqref{Eq-2.2} in Lemma B, for $f=l_{(A, B)}$, it is easy to see that 
	\begin{align}\label{Eq-2.10}
		d(0, \partial f(\mathbb{D}))=l_{(-A, -B)}(1).
	\end{align}
	Therefore, considering the function $f=l_{(A, B)}$ and $\omega(z)=z$ so that $|z|= R^{1,{\Phi(|\omega(z)|)}}_{A, B} $, in view of \eqref{Eq-2.8} and \eqref{Eq-2.10}, we see that
	\begin{align*}
	|\omega(z)|&+\sum_{n=2}^{\infty}|a_n||\omega(z)|^n+\Phi(|\omega(z)|)\sum_{n=2}^{\infty}|a_n|^2|\omega(z)|^{2n}\\&=\begin{cases}
		\Psi_1\left(R^{1,{\Phi(|\omega(z)|)}}_{A, B}\right)+\left(1+B\right)^{(A-B)/B},\; B\neq 0;\vspace{2mm}\\
		\Psi_2\left(R^{1,{\Phi(|\omega(z)|)}}_{A, B}\right)+e^{-A},\;\;\;\;\;\;\;\;\;\;\;\;\;\;\;\;\;\;\; B=0.
	\end{cases}\\&=\begin{cases}
	\left(1+B\right)^{(A-B)/B},\; B\neq 0;\vspace{2mm}\\e^{-A}\;\;\;\;\;\;\;\;\;\;\;\;\;\;\;\;\;\;\;\;\; B=0
	\end{cases}\\&= l_{-A, -B}(1)\\&= d(0, \partial f(\mathbb{D}))
	\end{align*}
	This shows that the radius $R^{1,\Phi(|\omega(z)|)}_{A, B}$ is sharp. This completes the proof.
\end{proof}
In mathematics, the Gaussian or ordinary hypergeometric function ${}_2F_1(a,b;c;z)$ is a special function represented by the hypergeometric series, that includes many other special functions as specific or limiting cases. It is a solution of a second-order linear ordinary differential equation (ODE).
The hypergeometric function is defined for $|z| <1$ by the power series 
\begin{align*}
	{}_2 F_1(a,b;c;z)=\sum_{n=0}^{\infty}\frac{(a)_n(b)_n}{(c)_n}\frac{z^n}{n!}=1+\frac{ab}{c}\frac{z}{1!}+\frac{a(a+1)b(b+1)}{c(c+1)}\frac{z^2}{2!}+\cdots .
\end{align*}
It is undefined (or infinite) if $c$ equals a nonpositive integer. Here $(q)_n$ is the Pochhammer symbol, which is defined by:
\begin{align*}
	(q)_n=
	\begin{cases}
		1\;\;\;\;\;\;\;\;\;\;\;\;\;\;\;\;\;\;\;\;\;\;\;\;\;\;\;\;\;\;\;\;\;\; n=0\\ q(q+1)\cdots(q+n-1)\;\;n>0
	\end{cases}
\end{align*} 
Bhowmik and Das (see \cite[Theorem 3, p. 1093]{Bhowmik-Das-JMAA-2018}) obtain the Bohr radius from the class $\mathcal{S}^*(\alpha)$, where $\alpha\in [0, 1/2]$. For $0\leq\alpha<1$, $A=1-2\alpha$ and $B=-1(\neq 0)$, Theorem \ref{Th-2.1} gives the sharp improved Bohr radius for the class of starlike functions of order $\alpha$. The following corollary improves the result \cite[Remark 3, p. 7]{Ahuja-Anand-jain-MDPI-2020} which is \cite[Corollary 1]{Anand-Jain-Kumar-BMMSS-2021}. 
\begin{cor}\label{Cor-2.7}
	Let $f(z)=z+\sum_{n=2}^{\infty}a_nz^n$ be in the class $\mathcal{ST}[A, B]$ and $\Phi : [0, 1]\to [0, \infty)$ be a continuous increasing function. Then 
	\begin{align*}
		|\omega(z)|+\sum_{n=2}^{\infty}|a_n||\omega(z)|^n+\Phi(|\omega(z)|)\sum_{n=2}^{\infty}|a_n|^2|\omega(z)|^{2n}\leq d(0, \partial f(\mathbb{D}))
	\end{align*}
	for $|z|=r\leq R^{\Phi(\omega(z)}_{\alpha}(r)$, where $ R^{\Phi(\omega(z)}_{\alpha}(r)$ is the root in $(0, 1)$ of the equations $	\Psi_\alpha(r)=0$, where
	\begin{align*}
		\Psi^{\Phi(\omega(z)}_\alpha(r):&=r(1-r)^{2(-1+\alpha)}+\Phi(r)\left(-1+{}_2 F_1(2-2\alpha,2-2\alpha,1,r^2)\right)r^2-2^{-2(1-\alpha)}.
	\end{align*}
	Each radius $ R^{\Phi(\omega(z)}_{\alpha}(r)$ is sharp. 
\end{cor}
\begin{rem}
	In particular, for $\alpha=0$, Corollary \ref{Cor-2.7} yields the sharp Bohr radius for the class of starlike functions which is $R_{0}(r)$ as the root of the equation
	\begin{align*}
		4r\left(1+2r-2r^3-r^4+\Phi(r)(4-3r^2+r^4)r^3\right)+(1-r^2)^3=0.
	\end{align*} 
	The radius $R^{\Phi(\omega(z)}_{0}(r)$ is best possible.
	
	Further, Corollary \ref{Cor-2.7} improves the result \cite[Corollary 1]{Anand-Jain-Kumar-BMMSS-2021} when $\alpha=0$. 
\end{rem}
\begin{rem}
	Putting $A=\beta$ and $B=-\beta(\neq 0)$, where $0<\beta\leq 1$ in Theorem \ref{Th-2.1}, we obtain the sharp improved Bohr radius for the class $\mathcal{ST}^{(\beta)}$. 
\end{rem}
\begin{cor}\label{Cor-2.8}
	$\mathcal{ST}[A, B]$ and $\Phi : [0, 1]\to [0, \infty)$ be a continuous increasing function. Then 
	\begin{align*}
		|\omega(z)|+\sum_{n=2}^{\infty}|a_n||\omega(z)|^n+\Phi(|\omega(z)|)\sum_{n=2}^{\infty}|a_n|^2|\omega(z)|^{2n}\leq d(0, \partial f(\mathbb{D}))
	\end{align*}
	for $|z|=r\leq R^{\Phi(\omega(z)}_{\beta,1}(r)$, where $ R^{\Phi(\omega(z)}_{\beta,1}(r)$ is the root in $(0, 1)$ of the equation $F^{\Phi(\omega(z)}_{\beta}(r)=0$, where 
	\begin{align*}
		F^{\Phi(\omega(z)}_{\beta}(r)=&-\beta^6\Phi(r)(1 + \beta)^2r^8+\beta^4 \left(3\Phi(r)+ 6\beta\Phi(r)+ \beta^2 (-1 +i 3\Phi(r))\right) r^6\\&\quad+\beta^4 (1 + \beta)^2 r^5+\beta^2 \left(2 \beta^3 + \beta (2 - 8\Phi(r)) + \beta^2 (7 - 4\Phi(r)) - 4\Phi(r)\right) r^4\\&\quad- \beta (2 + 7 \beta+ 2 \beta^2) r^2- (1 + \beta)^2 r+1.
	\end{align*}
	Each radius $ R^{\Phi(\omega(z))}_{\beta,1}(r)$ is sharp. 
\end{cor}
\begin{rem}
	It is important to note that $F_{\beta}(r)$ is a real-valued differentiable function satisfying
	\begin{align*}
		&F^{\Phi(\omega(z)}_{\beta}(0)=1>0\\& F^{\Phi(\omega(z)}_{\beta}(1)=-\beta(1+\beta)^2\left(4+\beta(-4\beta+(4-3\beta^2+\beta^4)\Phi(r))\right)<0
	\end{align*}
	which ensures the existence of the root $ R^{\Phi(\omega(z))}_{\beta,1}(r)$ in $(0, 1)$.  Further, Corollary \ref{Cor-2.8} improves the result \cite[Corollary 2]{Anand-Jain-Kumar-BMMSS-2021}. 
\end{rem}
\begin{rem}
		If $0<\beta\leq 1$, $A=\beta$ and $B=0$, then Theorem \ref{Th-2.1} yields the sharp improved Bohr radius for the class $\mathcal{ST}_{(\beta)}$.
\end{rem}
\begin{cor}\label{Cor-2.9}
	$\mathcal{ST}[A, B]$ and $\Phi : [0, 1]\to [0, \infty)$ be a continuous increasing function. Then 
	\begin{align*}
		|\omega(z)|+\sum_{n=2}^{\infty}|a_n||\omega(z)|^n+\Phi(|\omega(z)|)\sum_{n=2}^{\infty}|a_n|^2|\omega(z)|^{2n}\leq d(0, \partial f(\mathbb{D}))
	\end{align*}
	for $|z|=r\leq  R^{\Phi(\omega(z))}_{\beta,2}(r)$, where $ R_{\beta}(r)$ is the root in $(0, 1)$ of the equation $E_{\beta}(r)=0$, where
	\begin{align*}
	 R^{\Phi(\omega(z))}_{\beta,1}(r)=r-(1-e^{\beta r})r-\Phi(r)(1-\mathcal{J}_0(2\beta r))r^2-e^{-\beta}.
	\end{align*}
	Each radius $  R^{\Phi(\omega(z))}_{\beta,2}(r)$ is sharp.
\end{cor}
\begin{rem}
	It is important to note that $G_{\beta}$ is a real-valued differentiable function satisfying $ E^{\Phi(\omega(z))}_{\beta}(0)=-e^{-\beta}<0$ and 
	\begin{align*}
	 E^{\Phi(\omega(z))}_{\beta}(1)=\frac{-1+e^{2\beta}}{e^\beta}+\Phi(r)\left(-1+\mathcal{J}_0(2\beta)\right)>0\; \mbox{as}\;e^{2\beta}>1,\;\mathcal{J}_0(2\beta)>1\; \mbox{for}\; \beta>0
	\end{align*}
	which confirms the existence of the root $R_{\beta}(p)$ in $(0, 1)$.  Corollary \ref{Cor-2.9} improves the result \cite[Corollary 3]{Anand-Jain-Kumar-BMMSS-2021}. 
\end{rem}
\begin{rem}
	Letting $A=1$ and $B=(1-M)/M$, where $M>1/2$, Theorem \ref{Th-2.1} provides the following result for the class $\mathcal{ST}(M)$.
\end{rem}
\begin{cor}\label{Cor-2.10}
	If $M>1/2$ and $f(z)=z+\sum_{n=2}^{\infty}a_nz^n\in\mathcal{ST}(M)$, for $p\in\mathbb{N}$, 
	\begin{align*}
		|\omega(z)|+\sum_{n=2}^{\infty}|a_n||\omega(z)|^n+\Phi(|\omega(z)|)\sum_{n=2}^{\infty}|a_n|^2|\omega(z)|^{2n}\leq d(0, \partial f(\mathbb{D}))
	\end{align*}
	for $R^{\Phi(\omega(z))}_{M}(r)$, where $R^{\Phi(\omega(z))}_{M}(r)$ is the root in $(0, 1)$ of the equation 
	\begin{align*}
		&r+\sum_{n=2}^{\infty}\left(\prod_{k=0}^{n-2}\left(\frac{|(1-2M)+(1-M)k|}{M(k+1)}\right)\right)r^n+\Phi(r)\sum_{n=2}^{\infty}\left(\prod_{k=0}^{n-2}\left(\frac{|(1-2M)+(1-M)k|}{M(k+1)}\right)\right)^2r^{2n}\\&\quad-\left(2-\frac{1}{M}\right)^{\frac{(1-2M)}{M-1}}=0.
	\end{align*}
	The radius $R^{\Phi(\omega(z))}_{M}(r)$ is sharp.
\end{cor}

\begin{rem}
	Corollary \ref{Cor-2.10} is an improved version of the  result \cite[Corollary 4]{Anand-Jain-Kumar-BMMSS-2021} involving Schwartz function. 
\end{rem}
\begin{proof}[\bf Proof of Theorem \ref{Th-2.1}]
	Since $f(z)=z+\sum_{n=2}^{\infty}a_nz^n\in\mathcal{ST}[A, B]$, using the inequality \eqref{Eq-2.2}, we have \begin{align}\label{Eq-2.4}
		d(0, \partial f(\mathbb{D}))=\inf_{z\in\partial f(\mathbb{D})}|f(z)-f(0)|=\inf_{z\in\partial f(\mathbb{D})}|f(z)|\geq l_{(-A, -B)}(1).
	\end{align}
	\noindent Since $\omega(z)$ is a Schwarz function on $\mathbb{D}$ satisfying $|\omega(z)|\leq |z|$, a simple computation using Lemma A, Lemma B and the inequality \eqref{Eq-2.4} shows that  
	\begin{align*}
		&|f(re^{i\theta})|^p+|\omega(z)|+\sum_{n=2}^{\infty}|a_n(\omega(z))^n|\\&\leq \left(l_{(A, B)}(r)\right)^p+r+\sum_{n=2}^{\infty}\left(\prod_{k=0}^{n-2}\frac{|(B-A)+Bk|}{k+1}r^n\right)\\&\leq \begin{cases}
			\displaystyle\left(r(1+Br)^{(A-B)/B}\right)^p+r+\sum_{n=2}^{\infty}\left(\prod_{k=0}^{n-2}\frac{|(B-A)+Bk|}{k+1}r^n\right),\; B\neq 0,\vspace{2mm}\\
			 r^pe^{Apr}+re^{Ar},\;\;\;\;\;\;\;\;\;\;\;\;\;\;\;\;\;\;\;\;\;\;\;\;\;\;\;\;\;\;\;\;\;\;\;\;\;\;\;\;\;\;\;\;\;\;\;\;\;\;\;\;\;\;\;\;\;\;\;\;\;\; B=0
		\end{cases}
		\\&\leq l_{(-A, -B)}(1)\\&\leq d(0, \partial f(\mathbb{D}))
	\end{align*}
	for $r\leq R^p_{(A, B)}$, where $R^p_{(A, B)}\in (0, 1)$ is the root of the equations $\Psi_1(r)=0$ and $\Psi_2(r)=0$, where $\Psi_{j}$ $(j=3, 4)$ are given in the statement of the theorem. The desired inequality is thus obtained.\vspace{1.2mm}
	
	In order to show that the radius $ R^{2,{\Phi(|\omega(z)|)}}_{A, B}$, we consider the function $f=l_{(A, B)}$ which is given by \eqref{Eq-2.3}.
Therefore, considering the function $f=l_{(A, B)}$ and     
$\omega(z)=z$ so that $|z|= R^{2,{\Phi(|\omega(z)|)}}_{A, B} $, in view of \eqref{Eq-2.8} and \eqref{Eq-2.10}, we see that
\begin{align*}
	&|f(re^{i\theta})|^p+|\omega(z)|+\sum_{n=2}^{\infty}|a_n(\omega(z))^n|\\&=\begin{cases}
		\Psi_3\left(R^{1,{\Phi(|\omega(z)|)}}_{A, B}\right)+\left(1+B\right)^{(A-B)/B},\; B\neq 0;\vspace{2mm}\\
		\Psi_4\left(R^{1,{\Phi(|\omega(z)|)}}_{A, B}\right)+e^{-A},\;\;\;\;\;\;\;\;\;\;\;\;\;\;\;\;\;\;\; B=0.
	\end{cases}\\&=\begin{cases}
		\left(1+B\right)^{(A-B)/B},\; B\neq 0;\vspace{2mm}\\e^{-A}\;\;\;\;\;\;\;\;\;\;\;\;\;\;\;\;\;\;\;\;\; B=0
	\end{cases}\\&= l_{-A, -B}(1)\\&= d(0, \partial f(\mathbb{D}))
\end{align*}
This shows that the radius $R^{2,\Phi(|\omega(z)|)}_{A, B}$ is sharp. This completes the proof.
\end{proof}
Bhowmik and Das (see \cite[Theorem 3, p. 1093]{Bhowmik-Das-JMAA-2018}) obtain the Bohr radius from the class $\mathcal{S}^*(\alpha)$, where $\alpha\in [0, 1/2]$. For $0\leq\alpha<1$, $A=1-2\alpha$ and $B=-1(\neq 0)$, Theorem \ref{Th-2.1} gives the sharp improved Bohr radius for the class of starlike functions of order $\alpha$. The following corollary improves the result \cite[Remark 3, p. 7]{Ahuja-Anand-jain-MDPI-2020} which is \cite[Corollary 1]{Anand-Jain-Kumar-BMMSS-2021}. 
\begin{cor}\label{Cor-2.1}
	If $0\leq\alpha<1$ and $f(z)=z+\sum_{n=2}^{\infty}a_nz^n\in\mathcal{ST}[\alpha]$, for $p\in\mathbb{N}$, 
	\begin{align*}
		|f(re^{i\theta})|^p+|\omega(z)|+\sum_{n=2}^{\infty}|a_n(\omega(z))^n|\leq d(0, \partial f(\mathbb{D}))
	\end{align*}
	for $R_{\alpha}(p)$, where $R_{\alpha}(p)$ is the root of the equation 
	\begin{align*}
		r^p+\left(1-r\right)^{2(1+p)(1-\alpha)}-2^{2(1-\alpha)}r(1-r)^{2(1-\alpha)p}=0.
	\end{align*}
	The radius $R_{\alpha}(p)$ is sharp.
\end{cor}
\begin{rem}
	In particular, for $\alpha=0$, Corollary \ref{Cor-2.1} yields the sharp Bohr radius for the class of starlike functions which is $R_{0}(p)$ as the root of the equation
	\begin{align*}
		\left(1-r\right)^{2p+2}-4r(1-r)^{2p}+r^p=0.
	\end{align*} 
	Further, Corollary \ref{Cor-2.1} improves the result \cite[Corollary 1]{Anand-Jain-Kumar-BMMSS-2021} when $\alpha=0$. 
\end{rem}
\begin{rem}
	Putting $A=\beta$ and $B=-\beta(\neq 0)$, where $0<\beta\leq 1$ in Theorem \ref{Th-2.1}, we get the sharp improved Bohr radius for the class $\mathcal{ST}^{(\beta)}$. 
\end{rem}
\begin{cor}\label{Cor-2.2}
	If $0\leq\alpha<1$ and $f(z)=z+\sum_{n=2}^{\infty}a_nz^n\in\mathcal{ST}^{(\beta)}$, for $p\in\mathbb{N}$, 
	\begin{align*}
		|f(re^{i\theta})|^p+|\omega(z)|+\sum_{n=2}^{\infty}|a_n(\omega(z))^n|\leq d(0, \partial f(\mathbb{D}))
	\end{align*}
	for $R_{\beta}(p)$, where $R_{\alpha}(p)$ is the root in $(0, 1)$ of the equation $F_{\beta, p}(r)=0$, where
	\begin{align*}
		F_{\beta, p}(r):=(1+\beta)^2\left(r\left(1-\beta r\right)^{2(p-1)}+r^p\right)-\left(1-\beta r\right)^{2p}.
	\end{align*}
	The radius $R_{\beta}(p)$ is sharp.
\end{cor}
\begin{rem}
We note that $F_{\beta, p}$ is a real-valued differentiable function satisfying
	\begin{align*}
		F_{\beta, p}(0)=-1<0\;\; \mbox{and}\; F_{\beta, p}(1)=4\beta(1-\beta)^{2p-2}+(1+\beta)^2>0
	\end{align*}
	which ensures the existence of the root $R_{\beta}(p)$ in $(0, 1)$. Further, Corollary \ref{Cor-2.2} improves the result \cite[Corollary 2]{Anand-Jain-Kumar-BMMSS-2021}. 
\end{rem}
\begin{rem}
	If $0<\beta\leq 1$, $A=\beta$ and $B=0$, then Theorem \ref{Th-2.1} yields the sharp improved Bohr radius for the class $\mathcal{ST}_{(\beta)}$.
\end{rem}
\begin{cor}\label{Cor-2.3}
	If $0<\beta\leq 1$ and $f(z)=z+\sum_{n=2}^{\infty}a_nz^n\in \mathcal{ST}_{(\beta)}$, then for $p\in\mathbb{N}$, 
	\begin{align*}
		|f(re^{i\theta})|^p+|\omega(z)|+\sum_{n=2}^{\infty}|a_n(\omega(z))^n|\leq d(0, \partial f(\mathbb{D}))
	\end{align*}
	for $R_{\beta}(p)$, where $R_{\beta}(p)$ is the root in $(0, 1)$ of the equation $G_{\beta, p}(r)=0$, where 
	\begin{align*}
		G_{\beta, p}(r):=r^pe^{\beta pr}+re^{\beta r}-e^{-\beta}.
	\end{align*}
	The radius $R_{\beta}(p)$ is sharp.
\end{cor}
\begin{rem}
Note that $G_{\beta, p}$ is a real-valued differentiable function satisfying $G_{\beta, p}(0)=-e^{-\beta}<0$ and 
	\begin{align*}
		G_{\beta, p}(1)=e^{\beta p}+\frac{\left(e^{\beta}+1\right)\left(e^{\beta}-1\right)}{e^{\beta}}>0\; \mbox{as}\;e^{\beta}>1\; \mbox{for}\; \beta>0
	\end{align*}
	which confirms the existence of the root $R_{\beta}(p)$ in $(0, 1)$.  Corollary \ref{Cor-2.3} improves the result \cite[Corollary 3]{Anand-Jain-Kumar-BMMSS-2021}. 
\end{rem}
\begin{rem}
	Letting $A=1$ and $B=(1-M)/M$, where $M>1/2$, Theorem \ref{Th-2.1} provides the following result for the class $\mathcal{ST}(M)$.
\end{rem}
\begin{cor}\label{Cor-2.4}
	If $M>1/2$ and $f(z)=z+\sum_{n=2}^{\infty}a_nz^n\in\mathcal{ST}(M)$, then for $p\in\mathbb{N}$, 
	\begin{align*}
		|f(re^{i\theta})|^p+|\omega(z)|+\sum_{n=2}^{\infty}|a_n(\omega(z))^n|\leq d(0, \partial f(\mathbb{D}))
	\end{align*}
	for $R_{M}(p)$, where $R_{M}(p)$ is the root in $(0, 1)$ of the equation 
	\begin{align*}
		&\left(1+\frac{1-M}{M}r\right)^{\frac{(1-2M)p}{M-1}}r^p+r+\sum_{n=2}^{\infty}\left(\prod_{k=0}^{n-2}\left(\frac{|(1-2M)+(1-M)k|}{M(k+1)}\right)\right)r^n\\&\quad-\left(2-\frac{1}{M}\right)^{\frac{(1-2M)}{M-1}}=0.
	\end{align*}
	The radius $R_{M}(p)$ is sharp.
\end{cor}
\begin{rem}
	Corollary \ref{Cor-2.4} improves the result \cite[Corollary 4]{Anand-Jain-Kumar-BMMSS-2021} by adding a non-negative quantity in the Bohr inequality in terms of Schwartz function. 
\end{rem}
\begin{proof}[\bf Proof of Theorem \ref{Th-6.1}] 
	Since $f\in\mathcal{ST}[A, B]$, we have the estimate as in \eqref{Eq-2.4}.	Since $\omega(z)$ is a Schwarz function on $\mathbb{D}$ satisfying $|\omega(z)|\leq |z|=r<1$, an easy computation using Lemma A, Lemma B, \eqref{Eq-2.4} and \eqref{Eq-6.1} shows that  
	\begin{align*}
		&|\omega(z)|+\sum_{n=2}^{\infty}|a_n(\omega(z))^n|+\Phi(|\omega(z)|)\left(\frac{S_r}{\pi}\right)\\&\leq r+\sum_{n=2}^{\infty}\left(\prod_{k=0}^{n-2}\frac{|(B-A)+Bk|}{k+1}\right)r^n+\Phi(r)\left(\frac{S_r}{\pi}\right)\\&\leq \begin{cases}
			\displaystyle r+\sum_{n=2}^{\infty}\left(\prod_{k=0}^{n-2}\frac{|(B-A)+Bk|}{k+1}\right)r^n+\Phi(r)\sum_{n=1}^{\infty}n\left(\prod_{k=0}^{n-2}\frac{|(B-A)+Bk|}{k+1}\right)^2r^{2n},\; B\neq 0,\vspace{2mm}\\
			re^{Ar}+\Phi(r)\left(\mathcal{J}_0(2Ar)+Ar\mathcal{J}^{\prime}_0(2Ar)\right)r^2,\;\;\;\;\;\;\;\;\;\;\;\;\;\;\;\;\;\;\;\;\;\;\;\;\;\;\;\;\;\;\;\;\;\;\;\;\;\;\;\;\;\;\;\;\;\;\;\;\;\;\;\;\;\;\;\; B=0
		\end{cases}
		\\&\leq l_{(-A, -B)}(1)\\&\leq d(0, \partial f(\mathbb{D}))
	\end{align*}
	for $r\leq R^{2,{\Phi(|\omega(z)|)}}_{A, B}$, where $R^{2,{\Phi(|\omega(z)|)}}_{A, B}\in (0, 1)$ is the root of the equations $G_1(r)=0$ and $G_2(r)=0$, where $G_{j}$ $(j=1, 2)$ are given in the statement of the theorem. The desired inequality is thus obtained.\vspace{1.2mm}
	
	The next step of the proof is to show that the radius $R^{2,{\Phi(|\omega(z)|)}}_{A, B}$ is sharp. Hence, we consider the function $l_{(A, B)}$ which is given by \eqref{Eq-2.3}. By the similar method used in Theorem \ref{Th-2.2} and \eqref{Th-2.1}, it is easy to show that the radius $R^{2,{\Phi(|\omega(z)|)}}_{A, B}$ is sharp. Hence, we omit the details.
\end{proof}
\section{\bf The class $\mathcal{M}(\alpha)$ of $\alpha$-convex functions}
In $1969$, Mocanu (see \cite{Mocanu-1969}) introduced the class of $\alpha$-convex functions. For $\alpha\in\mathbb{R}$, a normalized analytic function of the form $f(z)=z+\sum_{n=2}^{\infty}a_nz^n$ is said to be $\alpha$-convex in $\mathbb{D}$ (or simply $\alpha$-convex) if $f$ satisfies 
\begin{align*}
\left(\frac{f(z)}{z}\right)f^{\prime}(z)\neq 0
\end{align*}
 and 
 \begin{align*}
 	{\rm Re}\left(\alpha\left(1+\frac{zf^{\prime\prime}(z)}{f^{\prime}(z)}\right)+(1-\alpha)\frac{zf^{\prime}(z)}{f(z)}\right)\geq 0,
 \end{align*}
 for all $z\in\mathbb{D}$. The class of such functions is denoted by $\mathcal{M}(\alpha)$. It is easy to see that $\mathcal{M}(0)=\mathcal{S}^*$, the class of starlike functions, and $\mathcal{M}(1)=\mathcal{C}$, the class of convex functions.\vspace{1.2mm}
 
 \subsection{\bf The refined Bohr inequality for the class $\mathcal{M}(\alpha)$ of $\alpha$-convex functions}
 In context of recent studies on Bohr inequality for different classes of functions here we study a refined version of the Bohr inequality for the class $\mathcal{M}(\alpha)$ and examine its sharpness.
\begin{thm}\label{Th-3.1}
If $\alpha>0$ and $f(z)=z+\sum_{n=2}^{\infty}a_nz^n\in\mathcal{M}(\alpha)$ and $\Phi : [0, 1]\to [0, \infty)$ be a continuous increasing function, then 
	\begin{align*}
		|\omega(z)|+\sum_{n=2}^{\infty}|a_n||\omega(z)|^n+\Phi(\omega(z))\sum_{n=2}^{\infty}|a_n|^2|\omega(z)|^{2n}\leq d(0, \partial f(\mathbb{D}))
	\end{align*}
	for $|z|=r\leq R^{\Phi(\omega(z))}_{1,\alpha}$, where $ R^{\Phi(\omega(z))}_{1,\alpha}$ is the root in $(0, 1)$ of the equation
	\begin{align*}
		r&+\sum_{n=2}^{\infty}\left(\sum\frac{\Upsilon(\alpha, q-1)c_1^{x_1}c_2^{x_2}\cdots c_n^{x_n}}{x_1!x_2!\cdots x_n!}\right)r^n\\&\quad+\Phi(r)\sum_{n=2}^{\infty}\left(\sum\frac{\Upsilon(\alpha, q-1)c_1^{x_1}c_2^{x_2}\cdots c_n^{x_n}}{x_1!x_2!\cdots x_n!}\right)^2r^{2n}\\&= k(-1, \alpha).
	\end{align*} 
	Each radius $ R^{\Phi(\omega(z))}_{1,\alpha}$ is sharp. 
\end{thm}
In particular, when $\Phi(\omega(z))=0$, we obtain the following corollary.
\begin{cor}\label{Cor-3.1}
If $\alpha>0$ and $f(z)=z+\sum_{n=2}^{\infty}a_nz^n\in\mathcal{M}(\alpha)$, then 
	\begin{align*}
		|\omega(z)|+\sum_{n=2}^{\infty}|a_n||\omega(z)|^n\leq d(0, \partial f(\mathbb{D}))
	\end{align*}
	for $|z|=r\leq  R^{0}_{1,\alpha}$, where $ R^{0}_{1,\alpha}$ is the root in $(0, 1)$ of the equation
	\begin{align*}
		r+\sum_{n=2}^{\infty}\left(\sum\frac{\Upsilon(\alpha, q-1)c_1^{x_1}c_2^{x_2}\cdots c_n^{x_n}}{x_1!x_2!\cdots x_n!}\right)r^n= k(-1, \alpha).
	\end{align*} 
	Each radius $ R^{0}_{1,\alpha}$ is sharp.
\end{cor}
\begin{rem}
 Corollary \ref{Cor-3.1} improves the result \cite[Theorem 3]{Anand-Jain-Kumar-BMMSS-2021} involving Schwartz functions. The radius $R^{0}_{1,\alpha}$ remains the same as the radius of \cite[Theorem 3]{Anand-Jain-Kumar-BMMSS-2021}. Additionally, when $\Phi(\omega(z))\neq0$, we obtain an improved and refined version of \cite[Theorem 3]{Anand-Jain-Kumar-BMMSS-2021}.
\end{rem}
\begin{rem}
In particular, when $\omega(z)=z$, Corollary \ref{Cor-3.1} is exactly the same as \cite[Theorem 3]{Anand-Jain-Kumar-BMMSS-2021}. Consequently, Corollary \ref{Cor-3.1} improves the result of \cite[Theorem 3]{Anand-Jain-Kumar-BMMSS-2021} for Schwartz functions.
\end{rem}
\begin{rem}
From Theorem \ref{Th-3.1}, we obtain the following immediate result which is a sharp refinement of \cite[Theorem 3]{Anand-Jain-Kumar-BMMSS-2021} involving Schwartz function. 
\end{rem}
\begin{cor}
Let $\alpha>0$ and $f(z)=z+\sum_{n=2}^{\infty}a_nz^n\in\mathcal{M}(\alpha)$. Then 
	\begin{align*}
		|\omega(z)|+\sum_{n=2}^{\infty}|a_n||\omega(z)|^n+\left(\frac{1}{2}+\frac{|\omega(z)|}{1-|\omega(z)|}\right)\sum_{n=2}^{\infty}|a_n|^2|\omega(z)|^{2n}\leq d(0, \partial f(\mathbb{D}))
	\end{align*}
	for $|z|=r\leq R^{|a_0|}_{\alpha}$, where $ R^{|a_0|}_{\alpha}$ is the root in $(0, 1)$ of the equation
	\begin{align*}
		r&+\sum_{n=2}^{\infty}\left(\sum\frac{\Upsilon(\alpha, q-1)c_1^{x_1}c_2^{x_2}\cdots c_n^{x_n}}{x_1!x_2!\cdots x_n!}\right)r^n\\&\quad+\left(\frac{1+r}{2(1-r)}\right)\sum_{n=2}^{\infty}\left(\sum\frac{\Upsilon(\alpha, q-1)c_1^{x_1}c_2^{x_2}\cdots c_n^{x_n}}{x_1!x_2!\cdots x_n!}\right)^2r^{2n}\\&= k(-1, \alpha).
	\end{align*} 
	Each radius $ R^{|a_0|}_{\alpha}$ is sharp.
\end{cor}
\subsection{\bf Improved Bohr inequality for the class $\mathcal{M}(\alpha)$ of $\alpha$-convex functions}
The recent improvement of the Bohr inequality by incorporating the term $S_r/\pi$ has garnered significant interest. Our next result provides an improved version of \cite[Theorem 3]{Anand-Jain-Kumar-BMMSS-2021} involving the term $S_r/\pi$ in the Bohr inequality for the class $\mathcal{M}(\alpha)$.
\begin{thm}\label{Th-6.2}
	Let $\alpha>0$ and $f(z)=z+\sum_{n=2}^{\infty}a_nz^n\in\mathcal{M}(\alpha)$ and $\Phi : [0, 1]\to [0, \infty)$ be a continuous increasing function. Then 
	\begin{align*}
		|\omega(z)|+\sum_{n=2}^{\infty}|a_n||\omega(z)|^n+\Phi(\omega(z))\left(\frac{S_r}{\pi}\right)\leq d(0, \partial f(\mathbb{D}))
	\end{align*}
	for $|z|=r\leq  R^{\Phi(\omega(z))}_{2,\alpha}$, where $R^{\Phi(\omega(z))}_{2,\alpha}$ is the root in $(0, 1)$ of the equation
	\begin{align*}
		r&+\sum_{n=2}^{\infty}\left(\sum\frac{\Upsilon(\alpha, q-1)c_1^{x_1}c_2^{x_2}\cdots c_n^{x_n}}{x_1!x_2!\cdots x_n!}\right)r^n\\&\quad+\Phi(r)\sum_{n=1}^{\infty}n\left(\sum\frac{\Upsilon(\alpha, q-1)c_1^{x_1}c_2^{x_2}\cdots c_n^{x_n}}{x_1!x_2!\cdots x_n!}\right)^2r^{2n}\\&= k(-1, \alpha).
	\end{align*} 
	Each radius $R^{\Phi(\omega(z))}_{2,\alpha}$ is sharp. 
\end{thm}
\begin{rem}
In particular when $\Phi(\omega(z))=0$, then we get improved version of \cite[Theorem 3]{Anand-Jain-Kumar-BMMSS-2021} involving Schwartz function and when $\Phi(\omega(z))\neq0$ then we get more improved and refined version of \cite[Theorem 3]{Anand-Jain-Kumar-BMMSS-2021}.
\end{rem}
\begin{rem}
	In particular, when $\omega(z)=z$, then Theorem \ref{Th-6.1} improves the result \cite[Theorem 3]{Anand-Jain-Kumar-BMMSS-2021} involving Schwartz function.
\end{rem}
To prove Theorem \ref{Th-3.1}, the following preliminary results will play crucial role. Below, Lemma C is for growth estimate which will help us to find the Euclidean distance $d(0, \partial f(\mathbb{D}))$, whereas, Lemma D is for coefficients bounds of functions for the class $\mathcal{M}(\alpha)$.
\begin{lemC}\emph{(see \cite[Theorem 7, p. 146]{Goodman-1983})}
	If $\alpha>0$, and $f(z)$ is $\alpha$-convex, then 
	\begin{align*}
		k(-r, \alpha)\leq |f(z)|\leq k(r, \alpha),\; z=re^{i\theta},
	\end{align*}
	where 
	\begin{align}\label{Eq-3.1}
		\displaystyle k(z, \alpha)=\left(\frac{1}{\alpha}\int_{0}^{z}\frac{\xi^{\frac{1}{\alpha}-1}}{\left(1-\xi\right)^{\frac{2}{\alpha}}}d\xi\right)^{\alpha}.
	\end{align}
	The inequalities are sharp for each $\alpha>0$ and each $r\in (0, 1)$. 
\end{lemC}
\begin{lemD}\emph{(see \cite[Theorem 2, p. 208]{Kulshrestha-1974})}
	Let $f(z)=z+\sum_{n=2}^{\infty}\in\mathcal{M}(\alpha)$ and let $S(n)$ be the set of all $n$-tuples $(x_1, \ldots, x_n)$ of non-negative integers for which $\sum_{i=1}^{n}ix_i=n$ and for each $n$-tuples defined by $\sum_{i=1}^{n}x_i=q$. If 
	\begin{align*}
		\Upsilon(\alpha, q)=\alpha(\alpha-1)(\alpha-2)\cdots(\alpha-q)
	\end{align*}
	with $\Upsilon(\alpha, 0)=\alpha$, then for $n=1, 2, \ldots$
	\begin{align}\label{Eq-3.2}
		|a_{n+1}|\leq\sum\frac{\Upsilon(\alpha, q-1)c_1^{x_1}c_2^{x_2}\cdots c_n^{x_n}}{x_1!x_2!\cdots x_n!},
	\end{align}
	where summation is taken over all $n$-tuples in $S(n)$ and 
	\begin{align*}
		c_n=\frac{1}{n!\alpha^n(1+n\alpha)}\prod_{k=0}^{n-1}(2+k\alpha).
	\end{align*}
	The result is sharp.
\end{lemD}
\begin{proof}[\bf Proof of Theorem \ref{Th-3.1}]
	Let $f\in\mathcal{M}(\alpha)$. Then in view of Lemma C, we have
	\begin{align*}
		k(-r, \alpha)\leq |f(z)|\leq k(r, \alpha),
	\end{align*}
	where $k$ is the function as given by \eqref{Eq-3.1}. Thus it follows that the Euclidean distance $d(0, \partial f(\mathbb{D}))$ satisfies   
	\begin{align}\label{Eq-3.3}
		d(0, \partial f(\mathbb{D}))\geq k(-1, \alpha).
	\end{align}
Moreover, the equality in \eqref{Eq-3.2} is achieved for the function $f(z)=k(z, \alpha)$.\vspace{1.2mm}
	 
	It is given that the Bohr radius $R_{\alpha}$ is the root in $(0, 1)$ of the equation 
	\begin{align*}
		r&+\sum_{n=2}^{\infty}\left(\sum\frac{\Upsilon(\alpha, q-1)c_1^{x_1}c_2^{x_2}\cdots c_n^{x_n}}{x_1!x_2!\cdots x_n!}\right)r^n\\&\quad+\Phi(r)\sum_{n=2}^{\infty}\left(\sum\frac{\Upsilon(\alpha, q-1)c_1^{x_1}c_2^{x_2}\cdots c_n^{x_n}}{x_1!x_2!\cdots x_n!}\right)^2r^{2n}=k(-1, \alpha)
	\end{align*}
	For $0\leq |z|=r\leq R^{\Phi(\omega(z))}_{1,\alpha}$. it is easy to verify that $k(r, \alpha)\leq k(-1, \alpha)$. Since $\omega : \mathbb{D}\to\mathbb{D}$ be such that $|\omega(z)|\leq |z|=r$, as the function $ \Phi$ is increasing, we have $\Phi(|\omega(z)|)\leq \Phi(r)$. Thus, using \eqref{Eq-3.1} and \eqref{Eq-3.3}, a simple computation using Lemma D gives that
	\begin{align*}
		|\omega(z)|&+\sum_{n=2}^{\infty}|a_n||\omega(z)|^n+\Phi(|\omega(z)|)\sum_{n=2}^{\infty}|a_n|^2|\omega(z)|^{2n}\\\leq r&+\sum_{n=2}^{\infty}\left(\sum\frac{\Upsilon(\alpha, q-1)c_1^{x_1}c_2^{x_2}\cdots c_n^{x_n}}{x_1!x_2!\cdots x_n!}\right)r^n\\&\quad+\Phi(r)\sum_{n=2}^{\infty}\left(\sum\frac{\Upsilon(\alpha, q-1)c_1^{x_1}c_2^{x_2}\cdots c_n^{x_n}}{x_1!x_2!\cdots x_n!}\right)^2r^{2n}\\&\leq k(-1, \alpha)\\&\leq d(f(0,\partial f(\mathbb{D}))
	\end{align*}
	for $|z|=r\leq R^{\Phi(\omega(z))}_{1,\alpha}$. Thus the desired inequality is established. \vspace{1.2mm}
	
	The next step of the proof is to show that the radius $R_{\alpha}$ is sharp. Hence, we consider the function $f(z)=k(z, \alpha)$ which is given by \eqref{Eq-3.2}. For $|z|=R^{\Phi(\omega(z))}_{\alpha}$, $\omega(z)=z$ and $f(z)=k(z, \alpha)$, in view of Lemma D and \eqref{Eq-3.3}, we obtain 
	\begin{align*}
		|\omega(z)|&+\sum_{n=2}^{\infty}|a_n||\omega(z)|^n+\Phi(|\omega(z)|)\sum_{n=2}^{\infty}|a_n|^2|\omega(z)|^{2n}\\= R^{\Phi(\omega(z))}_{1,\alpha}&+\sum_{n=2}^{\infty}\left(\sum\frac{\Upsilon(\alpha, q-1)c_1^{x_1}c_2^{x_2}\cdots c_n^{x_n}}{x_1!x_2!\cdots x_n!}\right)\left(R^{\Phi(\omega(z))}_{1,\alpha}\right)^n\\&\quad+\Phi\left(R^{\Phi(\omega(z))}_{1,\alpha}\right)\sum_{n=2}^{\infty}\left(\sum\frac{\Upsilon(\alpha, q-1)c_1^{x_1}c_2^{x_2}\cdots c_n^{x_n}}{x_1!x_2!\cdots x_n!}\right)^2\left(R^{\Phi(\omega(z))}_{1,\alpha}\right)^{2n}\\&= k(-1, \alpha)\\&=d(f(0,\partial f(\mathbb{D})).
	\end{align*}
	Thus the radius $R^{\Phi(\omega(z))}_{1,\alpha}$ is sharp.
\end{proof}
\begin{proof}[\bf Proof of Theorem \ref{Th-6.2}]
	Let $f\in\mathcal{M}(\alpha)$. In view of Lemma C, we have
	\begin{align*}
		k(-r, \alpha)\leq |f(z)|\leq k(r, \alpha),
	\end{align*}
	where $k$ is the function as given by \eqref{Eq-3.1}. It follows that the Euclidean distance $d(0, \partial f(\mathbb{D}))$ satisfies   
	\begin{align*}
		d(0, \partial f(\mathbb{D}))\geq k(-1, \alpha).
	\end{align*}
	Moreover, the equality in \eqref{Eq-3.2} is achieved for the function $f(z)=k(z, \alpha)$.\vspace{1.2mm}
	
	It is given that the Bohr radius $R_{\alpha}$ is the root in $(0, 1)$ of the equation 
	\begin{align*}
		r&+\sum_{n=2}^{\infty}\left(\sum\frac{\Upsilon(\alpha, q-1)c_1^{x_1}c_2^{x_2}\cdots c_n^{x_n}}{x_1!x_2!\cdots x_n!}\right)r^n\\&\quad+\Phi(r)\sum_{n=1}^{\infty}n\left(\sum\frac{\Upsilon(\alpha, q-1)c_1^{x_1}c_2^{x_2}\cdots c_n^{x_n}}{x_1!x_2!\cdots x_n!}\right)^2r^{2n}=k(-1, \alpha)
	\end{align*}
	For $0\leq |z|=r\leq R^{\Phi(\omega(z))}_{2,\alpha}$. it is easy to verify that $k(r, \alpha)\leq k(-1, \alpha)$. Since $\omega : \mathbb{D}\to\mathbb{D}$ be such that $|\omega(z)|\leq |z|=r$, as the function $ \Phi$ is increasing, we have $\Phi(|\omega(z)|)\leq \Phi(r)$. Thus, using \eqref{Eq-3.1} and \eqref{Eq-3.3}, a simple computation using Lemma D gives that
	\begin{align*}
		|\omega(z)|&+\sum_{n=2}^{\infty}|a_n||\omega(z)|^n+\Phi(|\omega(z)|)\sum_{n=1}^{\infty}n|a_n|^2|\omega(z)|^{2n}\\\leq r&+\sum_{n=2}^{\infty}\left(\sum\frac{\Upsilon(\alpha, q-1)c_1^{x_1}c_2^{x_2}\cdots c_n^{x_n}}{x_1!x_2!\cdots x_n!}\right)r^n\\&\quad+\Phi(r)\sum_{n=1}^{\infty}n\left(\sum\frac{\Upsilon(\alpha, q-1)c_1^{x_1}c_2^{x_2}\cdots c_n^{x_n}}{x_1!x_2!\cdots x_n!}\right)^2r^{2n}\\&\leq k(-1, \alpha)\\&\leq d(f(0,\partial f(\mathbb{D}))
	\end{align*}
	for $|z|=r\leq R^{\Phi(\omega(z))}_{2,\alpha}$. Thus the desired inequality is established. \vspace{1.2mm}
	
	The next step of the proof is to show that the radius $R^{\Phi(\omega(z))}_{2,\alpha}$ is sharp. Hence, we consider the function $f(z)=k(z, \alpha)$ which is given by \eqref{Eq-3.2}. For $|z|=R^{\Phi(\omega(z))}_{2,\alpha}$, $\omega(z)=z$ and $f(z)=k(z, \alpha)$, in view of Lemma D and \eqref{Eq-3.3}, we obtain 
	\begin{align*}
		|\omega(z)|&+\sum_{n=2}^{\infty}|a_n||\omega(z)|^n+\Phi(|\omega(z)|)n\sum_{n=1}^{\infty}|a_n|^2|\omega(z)|^{2n}\\= &R^{\Phi(\omega(z))}_{2,\alpha}+\sum_{n=2}^{\infty}\left(\sum\frac{\Upsilon(\alpha, q-1)c_1^{x_1}c_2^{x_2}\cdots c_n^{x_n}}{x_1!x_2!\cdots x_n!}\right)\left(R^{\Phi(\omega(z))}_{2,\alpha}\right)^n\\&\quad+\Phi\left(R^{\Phi(\omega(z))}_{2,\alpha}\right)\sum_{n=1}^{\infty}n\left(\sum\frac{\Upsilon(\alpha, q-1)c_1^{x_1}c_2^{x_2}\cdots c_n^{x_n}}{x_1!x_2!\cdots x_n!}\right)^2\left(R^{\Phi(\omega(z))}_{2,\alpha}\right)^{2n}\\&= k(-1, \alpha)\\&=d(f(0,\partial f(\mathbb{D})).
	\end{align*}
	Thus the radius $R^{\Phi(\omega(z))}_{2,\alpha}$ is sharp.
\end{proof}
\section{\bf The Bohr inequality for a class involving second-order differential subordination associated with Janowski functions}
For $\beta\geq\gamma\geq 0$, we consider the class $\mathcal{R}(\beta, \gamma, h)$ which is defined by making use of subordination as 
\begin{align*}
	\mathcal{R}(\beta, \gamma, h)=\{f\in\mathcal{A} : f(z)+\beta z f^{\prime}(z)+\gamma z^2 f^{\prime\prime}(z)\prec h(z),\; z\in\mathbb{D}\},
\end{align*}
where $h$ is a Janowski function. The class $\mathcal{R}(\beta, \gamma, h)$ can be seen as an extension to the class 
\begin{align*}
	\mathcal{R}(\beta, h)=\{f\in\mathcal{A} : f(z)+\beta z f^{\prime}(z)\prec h(z),\; z\in\mathbb{D}\}.
\end{align*}
It is easy to see that $\mathcal{R}(\beta, 0,  h)=\mathcal{R}(\beta, h)$. Many variation of this class have been studied for different geometric properties by various authors (see \cite{Yang-Liu-AAA-2011,Srivastava-Filomat-2016,Gao-Zhou-AMC-2007} and references there).\vspace{1.2mm}

For two analytic functions $f(z)=\sum_{n=0}^{\infty}a_nz^n$ and $g(z)=\sum_{n=0}^{\infty}b_nz^n$ for $z\in\mathbb{D}$, the Hadamard product (or convolution) is the function $f\star g$, defined by 
\begin{align*}
	(f\star g)(z)=\sum_{n=0}^{\infty}a_nb_nz^n,\; z\in\mathbb{D}.
\end{align*}
Consider the function $\varphi_{\lambda}$, defined by 
\begin{align*}
	\varphi_{\lambda}(z)=\int_{0}^{1}\frac{dt}{1-zt^{\lambda}}=\sum_{n=0}^{\infty}\frac{z^n}{1+\lambda n}.
\end{align*}
For ${\rm Re}\; \lambda\geq 0$, the function $\varphi_{\lambda}$ is convex in $\mathbb{D}$ due to \cite[Theorem 5, p. 113]{Ruscheweyh-PAMS-1975}.\vspace{1.2mm}

For $\beta\geq \gamma\geq 0$, let $\nu+\mu=\beta-\gamma$ and $\mu\nu=\gamma$ and 
\begin{align}\label{Eq-444.111}
	q(z)=\int_{0}^{1}\int_{0}^{1}h\left(zt^{\mu}s^{\nu}\right)dtds=(\varphi_{\nu}\star\varphi_{\mu})\star h(z).
\end{align}
Since $\varphi_{\nu}\star\varphi_{\mu}$ is a convex function and $h\in\mathcal{ST}[A, B]$, it follows from \cite[Theorem 5, p. 167]{Ma-Minda-1992} and $q\in\mathcal{ST}[A, B]$.\vspace{1.2mm}

As a follow-up to recent studies on the Bohr inequality, we present our next result, which refines the Bohr inequality \cite[Theorem 2]{Anand-Jain-Kumar-BMMSS-2021} for the class $\mathcal{R}(\beta, \gamma, h)$ by incorporating a Schwartz function.
\begin{thm}\label{Th-4.1}
	Let $-1\leq B\leq A\leq 1$, $f(z)=\sum_{n=0}^{\infty}a_nz^n\in\mathcal{R}(\beta, \gamma, h)$ and $h$ be a Janowski starlike function, and and $\Phi : [0, 1]\to [0, \infty)$ be a continuous increasing function. Then 
	\begin{align*}
		\sum_{n=1}^{\infty}|a_n|\omega(z)|^n+\frac{\Phi(|\omega(z)|)}{|h^{\prime}(0)|}\sum_{n=2}^{\infty}|a_n|^2\omega(z)|^{2n}\leq d(h(0), \partial h(\mathbb{D}))
	\end{align*}
	for $|z|=r\leq R^{\nu, \mu}_{A, B}$, where $R^{\nu, \mu}_{A, B}$ is the smallest root of the equation
	\begin{align}\label{Eq-44.11}
		&\frac{r}{1+(\nu+\mu)+\nu\mu}+\sum_{n=2}^{\infty}\left(\frac{\displaystyle\prod_{k=0}^{n-2}\frac{|(B-A)+Bk|}{k+1}}{1+(\nu+\mu)n+\nu\mu n^2}\right)r^n\\&\quad\nonumber+{\Phi(r)}\sum_{n=2}^{\infty}\left(\frac{\displaystyle\prod_{k=0}^{n-2}\frac{|(B-A)+Bk|}{k+1}}{1+(\nu+\mu)n+\nu\mu n^2}\right)^2r^{2n}=l_{(-A, -B)}(1).
	\end{align}
	Each radius $ R^{\nu, \mu}_{A, B} $ is sharp. 
\end{thm}
\begin{rem}
	In particular, if $\Phi(\omega(z))=0$, we have an improved version of \cite[Theorem 2]{Anand-Jain-Kumar-BMMSS-2021} in terms of Schwartz function. On the other hand, if $\Phi(\omega(z))\neq0$, we obtain a more improved version of \cite[Theorem 2]{Anand-Jain-Kumar-BMMSS-2021} that involves a general form of positive function $\Phi(\omega(z))$.
\end{rem}
\begin{rem}
	In particular, when $\omega(z)=z,$ we obtain a more general version of \cite[Theorem 2]{Anand-Jain-Kumar-BMMSS-2021} in terms of a positive function $\omega(z)$.
\end{rem}
\begin{proof}[\bf Proof of the Theorem \ref{Th-4.1}]
	Suppose that $G(z)=f(z)+\beta z f^{\prime}(z)+\gamma z^2 f^{\prime\prime}(z)\prec h(z)$,\; $a_1=1$.\vspace{1.2mm}
	
Consider 
\begin{align*}
	\frac{1}{h^{\prime}(0)}\sum_{n=1}^{\infty}[1+\beta n+\gamma n(n-1)]a_nz^n=\frac{G(z)-G(0)}{h^{\prime}(0)}\prec \frac{h(z)-h(0)}{h^{\prime}(0)}=H(z).
\end{align*}
Since $h\in\mathcal{ST}[A, B]$, it follows that $H\in\mathcal{ST}[A, B]$. Thus by Lemma A gives that 
\begin{align*}
|a_n|\leq 	\bigg|\frac{h^{\prime}(0)}{1+\beta n+\gamma n(n-1)}\bigg|\prod_{k=0}^{n-2}\frac{|(B-A)+Bk|}{k+1}
\end{align*}
for each $n\geq 2$.\vspace{1.2mm} 

Given the inequality above, as well as the condition that $|\omega(z)|\leq |z|=r$ and $\Phi$ being an increasing function, a straightforward computation yields that
\begin{align}\label{Eq-4.1}
&\sum_{n=1}^{\infty}|a_n|\omega(z)|^n+\frac{\Phi(|\omega(z)|)}{|h^{\prime}(0)|}\sum_{n=2}^{\infty}|a_n|^2\omega(z)|^{2n}\\&\leq\nonumber \frac{|h^{\prime}(0)|r}{1+(\nu+\mu)+\nu\mu}+\sum_{n=2}^{\infty}\left(\frac{|h^{\prime}(0)|\displaystyle\prod_{k=0}^{n-2}\frac{|(B-A)+Bk|}{k+1}}{1+(\nu+\mu)n+\nu\mu n^2}\right)r^n\\&\nonumber\quad+\frac{\Phi(r)}{|h^{\prime}(0)|}\sum_{n=2}^{\infty}\left(\frac{\displaystyle |h^{\prime}(0)|\prod_{k=0}^{n-2}\frac{|(B-A)+Bk|}{k+1}}{1+(\nu+\mu)n+\nu\mu n^2}\right)^2r^{2n}
\end{align}
Since the function $H\in\mathcal{ST}[A, B]$, using \eqref{Eq-2.2} and \eqref{Eq-2.3}, the following inequality holds
\begin{align}\label{Eq-4.2}
	l_{(-A, -B)}(r)\leq |H\left(re^{i\theta}\right)|\leq l_{(A, B)}(r),\; 0<r\leq 1. 
\end{align}
It follows from \eqref{Eq-4.2} that
\begin{align*}
	d(0, \partial H(\mathbb{D}))\geq l_{(-A, -B)}(1)
\end{align*}
which gives that 
\begin{align*}
	d(h(0), \partial h(\mathbb{D}))=\inf_{z\in\partial f(\mathbb{D})}|h(\xi)-h(0)|\geq |h^{\prime}(0)|l_{(-A, -B)}(1).
\end{align*}
Thus we have 
\begin{align}\label{Eq-4.3}
	|h^{\prime}(0)|\leq \frac{d(h(0), \partial h(\mathbb{D}))}{l_{(-A, -B)}(1)}.
\end{align}
Using \eqref{Eq-4.1} and \eqref{Eq-4.3}, we obtain 
\begin{align*}
&\sum_{n=1}^{\infty}|a_n|\omega(z)|^n+\frac{\Phi(|\omega(z)|)}{|h^{\prime}(0)|}\sum_{n=2}^{\infty}|a_n|^2\omega(z)|^{2n}\\&\leq\nonumber \frac{d(h(0), \partial h(\mathbb{D}))}{l_{(-A, -B)}(1)}\left(\frac{r}{1+(\nu+\mu)+\nu\mu}+\sum_{n=2}^{\infty}\left(\frac{\displaystyle\prod_{k=0}^{n-2}\frac{|(B-A)+Bk|}{k+1}}{1+(\nu+\mu)n+\nu\mu n^2}\right)r^n\right)\\&\nonumber\quad+\frac{d(h(0), \partial h(\mathbb{D}))}{l_{(-A, -B)}(1)}\Phi(r)\sum_{n=2}^{\infty}\left(\frac{\displaystyle \prod_{k=0}^{n-2}\frac{|(B-A)+Bk|}{k+1}}{1+(\nu+\mu)n+\nu\mu n^2}\right)^2r^{2n}\\&\leq d(h(0), \partial h(\mathbb{D}))
\end{align*}
for $|z|=r\leq R^{\nu, \mu}_{A, B}$, where the radius $R^{\nu, \mu}_{A, B}$ is the smallest positive root of the equation \eqref{Eq-44.11}.\vspace{1.2mm} 

To show that the radius $ R^{\nu, \mu}_{A, B} $ is sharp, we consider the function $f(z)=q(z)=(\varphi_{\nu}\star\varphi_{\mu})\star h(z)$ as defined in \eqref{Eq-444.111} with $h(z)=l_{(A, B)}(z)$, where $l_{(A, B)}(z)$ is given by \eqref{Eq-2.3}. Moreover, $f(z)\in\mathcal{R}(\beta, \gamma, h)$. This yields that 
\begin{align}\label{Eq-4.6}
	f(z)=\frac{z}{1+(\nu+\mu)+\nu\mu}+\sum_{n=2}^{\infty}\left(\frac{\displaystyle\prod_{k=0}^{n-2}\frac{|(B-A)+Bk|}{k+1}}{1+(\nu+\mu)n+\nu\mu n^2}\right)z^n.
\end{align}
Therefore, for $|z|=R^{\nu, \mu}_{A, B}$, for $f(z)$ as given in \eqref{Eq-4.6} and $\omega(z)=z$, we see that 
\begin{align*}
	&\sum_{n=1}^{\infty}|a_n|\omega(z)|^n+\Phi(|\omega(z)|)\sum_{n=2}^{\infty}|a_n|^2\omega(z)|^{2n}\\&=|h^{\prime}(0)|\left(\frac{r}{1+(\nu+\mu)+\nu\mu}+\sum_{n=2}^{\infty}\left(\frac{\displaystyle\prod_{k=0}^{n-2}\frac{|(B-A)+Bk|}{k+1}}{1+(\nu+\mu)n+\nu\mu n^2}\right)r^n\right)\\&\quad\nonumber+|h^{\prime}(0)|{\Phi(r)}\sum_{n=2}^{\infty}\left(\frac{\displaystyle\prod_{k=0}^{n-2}\frac{|(B-A)+Bk|}{k+1}}{1+(\nu+\mu)n+\nu\mu n^2}\right)^2r^{2n}\\&=|h^{\prime}(0)|l_{(-A, -B)}(1)\\&=d(h(0), \partial h(\mathbb{D})).
\end{align*}
Thus, the radius $ R^{\nu, \mu}_{A, B} $ is sharp. 
\end{proof}
\section{\bf The Bohr inequality for typically real functions $\mathcal{TR}$}
The class of typically real functions was introduced by Rogosinski (see \cite{Rogosinski-MZ-1932}). An analytic function $f(z)=z+\sum_{n=2}^{\infty}a_nz^n$ which satisfies the condition ${\rm Im}\; (f(z)){\rm Im}\; z>0$ for non-real $z\in\mathbb{D}$ is said to be typically real in $\mathbb{D}$. The class of such functions is denoted by $\mathcal{TR}$. In this section, we study the Bohr radius, the Bohr-Rogosinski radius for the class $\mathcal{TR}$ of typically real functions.
\subsection{\bf Refined Bohr inequality for typically real functions $\mathcal{TR}$}

Our next result is about the refinement of Bohr inequality \cite[Theorem 4]{Anand-Jain-Kumar-BMMSS-2021} for the class $\mathcal{TR}$ involving Schwartz function.
\begin{thm}\label{Th-5.1}
	Let $f(z)=z+\sum_{n=2}^{\infty}a_nz^n$ be in the class $\mathcal{TR}$ and $z=re^{i\theta}\in\mathbb{D}$ and and $\Phi : [0, 1]\to [0, \infty)$ be a continuous increasing function. Then 
	\begin{align*}
		|\omega(z)|+\sum_{n=2}^{\infty}|a_n||\omega(z)|^n+\Phi(|\omega(z)|)\sum_{n=2}^{\infty}|a_n|^2|\omega(z)|^{2n}\leq d(0, \partial f(\mathbb{D}))
	\end{align*}
	for $|z|\leq R_1^{\Phi(\omega(z)}$, where $ R_1^{\Phi(\omega(z)}$ is the unique root in $(0, 1)$ of the equation $F(r)=0$, where
	\begin{align}\label{Eqn-5.1}
		F(r):=\frac{r}{(1-r)^2}+\Phi(r)\frac{r^4\left(r^4-3r^2+4\right)}{1-3r^2+3r^4-r^6}-\frac{1}{4}.
	\end{align} 
	The radius $ R_1^{\Phi(\omega(z)}$ is sharp. 
\end{thm}
\begin{table}[ht]
	\centering
	\begin{tabular}{|l|l|}
		\hline
		$\Phi(\omega(z))$& $R_1^{\Phi(\omega(z)} $\\
		\hline
		$\omega(z)$& $0.17126$\\
		\hline
		$ e^\omega(z)$& $0.16953$\\
		\hline
		$ \sin \omega(z)$& $0.17126$ \\
		\hline
		$\log(1+\omega(z))$& $0.17129$ \\
		\hline
		$\frac{1}{2}+\frac{\omega(z)}{1-\omega(z)}$& $0.17033$ \\ 
		\hline
	\end{tabular}\vspace{2.5mm}
	\caption{The table exhibits the roots $R^{\Phi(|\omega(z)|)}_{1}$ of the equation \eqref{Eqn-5.1} for different choice of $\Phi(\omega(z))=\omega(z)$, $ e^{\omega(z)} $, $\sin(\omega(z))$, $\log(1+\omega(z))$ and $\frac{1}{2}+\frac{\omega(z)}{1-\omega(z)}$ with $\omega(z)=z$.}
\end{table}
\begin{rem}
	In particular, when $\Phi(|\omega(z)|)=0$, an improved version of \cite[Theorem 4]{Anand-Jain-Kumar-BMMSS-2021} is obtained involving Schwartz functions. On the other hand, when $\Phi(|\omega(z)|)\neq0$, a further improved version of \cite[Theorem 4]{Anand-Jain-Kumar-BMMSS-2021} is obtained.
\end{rem}
\begin{rem}
	In particular, when $\omega(z)=z$, we see that Theorem \ref{Th-5.1} reduces exactly to \cite[Theorem 4]{Anand-Jain-Kumar-BMMSS-2021}. Clearly, Theorem \ref{Th-5.1} improves the result \cite[Theorem 4]{Anand-Jain-Kumar-BMMSS-2021}.
\end{rem}
\subsection{\bf Improved Bohr inequality for typically real functions $\mathcal{TR}$}

In our next result, we improve the Bohr inequality \cite[Theorem 4]{Anand-Jain-Kumar-BMMSS-2021} for the class $\mathcal{TR}$ by incorporating the positive term $S_r/\pi$ in terms of Schwartz function.
\begin{thm}\label{Th-6.3}
	Let $f(z)=z+\sum_{n=2}^{\infty}a_nz^n$ be in the class $\mathcal{TR}$ and $z=re^{i\theta}\in\mathbb{D}$ and and $\Phi : [0, 1]\to [0, \infty)$ be a continuous increasing function. Then 
	\begin{align*}
		|\omega(z)|+\sum_{n=2}^{\infty}|a_n||\omega(z)|^n+\Phi(|\omega(z)|)\left(\frac{S_r}{\pi}\right)\leq d(0, \partial f(\mathbb{D}))
	\end{align*}
	for $|z|=r\leq  R_2^{\Phi(\omega(z)}$, where $R_2^{\Phi(\omega(z)}$ is the unique root in $(0, 1)$ of the equation $F(r)=0$, where
	\begin{align}\label{Eq-55.22}
		F(r):=\frac{r}{(1-r)^2}+\Phi(r)\frac{r^2\left(1+4r^2+r^4\right)}{(1-r^2)^4}-\frac{1}{4}.
	\end{align} 
	The radius $R_2^{\Phi(\omega(z)}$ is sharp. 
\end{thm}
\begin{table}[ht]
	\centering
	\begin{tabular}{|l|l|}
		\hline
	$\Phi(\omega(z))$& $R_2^{\Phi(\omega(z)} $\\
	\hline
	$\omega(z)$& $0.16864$\\
	\hline
	$ e^\omega(z)$& $0.15459$\\
	\hline
	$ \sin \omega(z)$& $0.16865$ \\
	\hline
	$\log(1+\omega(z))$& $0.16885$ \\
	\hline
	$\frac{1}{2}+\frac{\omega(z)}{1-\omega(z)}$& $0.16069$ \\ 
	\hline
	\end{tabular}\vspace{2.5mm}
	\caption{The table exhibits the roots $R^{\Phi(|\omega(z)|)}_{2}$ of the equation \eqref{Eq-55.22} for different choice of $\Phi(\omega(z))=\omega(z)$, $ e^{\omega(z)} $, $\sin(\omega(z))$, $\log(1+\omega(z))$ and $\frac{1}{2}+\frac{\omega(z)}{1-\omega(z)}$ with $\omega(z)=z$.}
\end{table}
\begin{rem}
	In particular when $\Phi(|\omega(z)|)=0$, then we get an improved version of  \cite[Theorem 4]{Anand-Jain-Kumar-BMMSS-2021} involving Schwartz function and when $\Phi(|\omega(z)|)\neq0$ then we get a more improvement version including the term $S_r/\pi$ of \cite[Theorem 4]{Anand-Jain-Kumar-BMMSS-2021}.
\end{rem}
\begin{rem}
	In particular, when $\omega(z)=z$, we see that Theorem \ref{Th-6.3} reduces exactly to \cite[Theorem 4]{Anand-Jain-Kumar-BMMSS-2021}. Clearly, Theorem \ref{Th-6.3} improves the result \cite[Theorem 4]{Anand-Jain-Kumar-BMMSS-2021}.
\end{rem}
\begin{lemE}\emph{(see \cite[Theorem 3, p. 185]{Goodman-1983})}
	If $f(z)$ is in $\mathcal{TR}$ and $z=re^{i\theta}\in\mathbb{D}$, then the coefficients satisfy the inequality 
	\begin{align}\label{Eq-5.1}
		m_n\leq |a_n|\leq n,
	\end{align}
	where $m_n=\min\{\sin n\theta/\sin\theta\}$ for each $n$. The inequality is sharp for each $n$. 
\end{lemE}
\begin{lemF}\emph{(see \cite[Theorem 1, p. 136]{Remizova-1963})}
	Let the function $f$ be in $\mathcal{TR}$ and $z=re^{i\theta}\in\mathbb{D}$. Then 
	\begin{align*}
		|f(z)|\geq\begin{cases}
			\displaystyle\bigg|\frac{z}{(1+z)^2}\bigg|, \quad\quad\quad\;\;\;\mbox{if}\; {\rm Re}\left(\frac{1+z^2}{z}\right)\geq 2;\vspace{2mm}\\
			\displaystyle\frac{r\left(1-r^2\right)|\sin\theta|}{|1-z^2|^2},\quad\mbox{if}\; {\rm Re}\left(\frac{1+z^2}{z}\right)\leq 2;\vspace{2mm}\\
			\displaystyle\bigg|\frac{z}{(1-z)^2}\bigg|, \quad\quad\quad\;\;\;\mbox{if}\; {\rm Re}\left(\frac{1+z^2}{z}\right)\leq -2.
		\end{cases}
	\end{align*}
	The result is sharp.
\end{lemF}
\begin{proof}[\bf Proof of the Theorem \ref{Th-5.1}]
First of all we assume that $f(z)=z+\sum_{n=2}^{\infty}a_nz^n\in\mathcal{TR}$. Then by the Lemma F and the proof of \cite[Theorem 2]{Anand-Jain-Kumar-BMMSS-2021} (see also \cite{Brannan-GMJ-1969}), it follows that the distance between the origin and the boundary of $f(\mathbb{D})$ satisfies the inequality
\begin{align}\label{Eq-5.2}
|f(z)|\geq \frac{1}{4},\; z\in\mathbb{D}.
\end{align}
Using the inequality \eqref{Eq-5.2}, we see that 
\begin{align}\label{Eq-5.3}
d(0, \partial f(\mathbb{D}))=\inf_{z\in\partial f(\mathbb{D})}|f(\xi)|\geq \frac{1}{4}.
\end{align}
We see that the function $F$ as defined in \eqref{Eq-55.22} is a continuous real valued function on $[0, 1]$ satisfying
\begin{align*}
	F(0)=-\frac{1}{4}\; \mbox{and}\; \lim\limits_{r\to 1^{-}}F(r)+\infty.
\end{align*}
Moreover, since $\Phi$ is an increasing function, it can be shown that $F$ is also increasing on $[0, 1]$ as 
\begin{align*}
	\frac{r}{(1-r)^2}\; \mbox{and}\; \frac{r^4\left(r^4-3r^2+4\right)}{1-3r^2+3r^4-r^6}
\end{align*}
are increasing on $[0, 1]$. Hence, $R_1^{\Phi(\omega(z)}$ is the unique root in $(0, 1)$ of $F(r)=0$.\vspace{1.2mm}

By the inequalities \eqref{Eq-5.1} and \eqref{Eq-5.3}, in view of $|\omega(z)|\leq |z|=r$, a computation shows that
\begin{align*}
	|\omega(z)|&+\sum_{n=2}^{\infty}|a_n||\omega(z)|^n+\Phi(|\omega(z)|)\sum_{n=2}^{\infty}|a_n|^2|\omega(z)|^{2n}\\&\leq r+\sum_{n=2}^{\infty}nr^n+\Phi(r)\sum_{n=2}^{\infty}n^2r^{2n}\\&=\frac{r}{(1-r)^2}+\Phi(r)\frac{r^4\left(r^4-3r^2+4\right)}{1-3r^2+3r^4-r^6}\\&\leq\frac{1}{4}\\& \leq d(0, \partial f(\mathbb{D}))
\end{align*}
if $r<R_1^{\Phi(\omega(z)}$. Thus the desired inequality is obtained.\vspace{1.2mm}

To show that the radius $ R_1^{\Phi(\omega(z)}$ is sharp, for $-\pi<t\leq \pi$, we consider the function $l_t : \mathbb{D}\to\mathbb{C}$ defined by 
\begin{align}\label{Eq-5.5}
	l_t(z)=\frac{z}{1-2z\cos t+z^2},\; z\in\mathbb{D}.
\end{align}
For $|z|=R_1^{\Phi(\omega(z)}$, $\omega(z)=z$ and $t=0$, in view of $f(z)=l_t$, we see that
\begin{align*}
	|\omega(z)|&+\sum_{n=2}^{\infty}|a_n||\omega(z)|^n+\Phi(|\omega(z)|)\sum_{n=2}^{\infty}|a_n|^2|\omega(z)|^{2n}\\&= R_1^{\Phi(\omega(z)}+\sum_{n=2}^{\infty}n\left(R_1^{\Phi(\omega(z)}\right)^n+\Phi\left(R_1^{\Phi(\omega(z)}\right)\sum_{n=2}^{\infty}|a_n|^2\left(R_1^{\Phi(\omega(z)}\right)^{2n}\\&=\frac{R_1^{\Phi(\omega(z)}}{\left(1-R_1^{\Phi(\omega(z)}\right)^2}+\Phi\left(R_1^{\Phi(\omega(z)}\right)\frac{\left(R_1^{\Phi(\omega(z)}\right)^4\left(\left(R_1^{\Phi(\omega(z)}\right)^4-3\left(R_1^{\Phi(\omega(z)}\right)^2+4\right)}{1-3\left(R_1^{\Phi(\omega(z)}\right)^2+3\left(R_1^{\Phi(\omega(z)}\right)^4-\left(R_1^{\Phi(\omega(z)}\right)^6}\\&=\frac{1}{4}\\&=d(0, \partial f(\mathbb{D})).
\end{align*}
This shows that the radius $R_1^{\Phi(\omega(z)}$ is sharp. 
\end{proof}
\begin{proof}[\bf Proof of the Theorem \ref{Th-6.3}]
	By the inequalities \eqref{Eq-5.1} and \eqref{Eq-5.3}, in view of $|\omega(z)|\leq |z|=r$, a computation shows that
	\begin{align*}
		|\omega(z)|&+\sum_{n=2}^{\infty}|a_n||\omega(z)|^n+\Phi(|\omega(z)|)\left(\frac{S_r}{\pi}\right)\\&\leq r+\sum_{n=2}^{\infty}nr^n+\Phi(r)\sum_{n=1}^{\infty}n^3r^{2n}\\&=\frac{r}{(1-r)^2}+\Phi(r)\frac{r^2\left(1+4r^2+r^4\right)}{(1-r^2)^4}\\&\leq\frac{1}{4}\\& \leq d(0, \partial f(\mathbb{D}))
	\end{align*}
	if $r\leq R_2^{\Phi(\omega(z)}$. Thus the desired inequality is obtained.\vspace{1.2mm}
	
	To show that the radius $ R_2^{\Phi(\omega(z)}$ is sharp, for $-\pi<t\leq \pi$, we consider the function $l_t : \mathbb{D}\to\mathbb{C}$ defined by 
	\begin{align*}
		l_t(z)=\frac{z}{1-2z\cos t+z^2},\; z\in\mathbb{D}.
	\end{align*}
	For $|z|=R_2^{\Phi(\omega(z)}$, $\omega(z)=z$ and $t=0$, in view of $f(z)=l_t$, we see that
	\begin{align*}
		|\omega(z)|&+\sum_{n=2}^{\infty}|a_n||\omega(z)|^n+\Phi(|\omega(z)|)\left(\frac{S_r}{\pi}\right)\\&= R_2^{\Phi(\omega(z)}+\sum_{n=2}^{\infty}n\left(R_2^{\Phi(\omega(z)}\right)^n+\Phi\left(R_2^{\Phi(\omega(z)}\right)\sum_{n=1}^{\infty}n|a_n|^2\left(R_{\Phi(r)}\right)^{2n}\\&=\frac{R_2^{\Phi(\omega(z)}}{\left(1-R_2^{\Phi(r)}\right)^2}+\Phi\left(R_2^{\Phi(\omega(z)}\right)\frac{\left(R_2^{\Phi(\omega(z)}\right)^2\left(1+4\left(R_2^{\Phi(\omega(z)}\right)^2+\left(R_2^{\Phi(\omega(z)}\right)^4\right)}{\left(1-\left(R_2^{\Phi(\omega(z)}\right)^2\right)^4}\\&=\frac{1}{4}\\&=d(0, \partial f(\mathbb{D})).
	\end{align*}
	This shows that the radius $R_2^{\Phi(\omega(z)}$ is sharp. 
\end{proof}

\section{\bf Declaration}
\noindent\textbf{Compliance of Ethical Standards:}\\

\noindent\textbf{Conflict of interest:} The authors declare that there is no conflict  of interest regarding the publication of this paper.\vspace{1.5mm}

\noindent\textbf{Data availability statement:}  Data sharing not applicable to this article as no datasets were generated or analysed during the current study.


\begin{thebibliography}{200}
	
\bibitem{Ahamed-CVEE-2022} {\sc M. B. Ahamed} , The sharp refined Bohr–Rogosinski inequalities for certain classes of harmonic mappings, \textit{Complex Var. Elliptic Equ.}, {\bf 69}(4)(2022), 586–606.	
		
\bibitem{Ahamed-CVEE-2021} {\sc M. B. Ahamed, V. Allu} and {\sc H. Halder}, Improved Bohr inequalities for certain class of harmonic univalent functions, \textit{Complex Var. Elliptic Equ.}, {\bf 68}(2)(2021), 267–290.		
		
\bibitem{Aha-Aha-CMFT-2023} {\sc M. B. Ahamed} and {\sc  S. Ahammed},Bohr Type Inequalities for the Class of Self-Analytic Maps on the Unit Disk, \textit{Comput. Methods Funct. Theory} {\bf 23}, 171(2023), 789-806.
	
\bibitem{Ahamed-Ahamed-MJM-2024} {\sc M.B. Ahamed}, and {\sc S. Ahamed}, Bohr inequalities for certain classes of harmonic happings, \textit{Mediterr. J. Math.}, {\bf 21:21}(2024).
	                                     
\bibitem{Ahamed-Allu-RMJ-2022} {\sc M.B. Ahamed}, and {\sc V. Allu}, Bohr phenomenon for certain classes of harmonic mappings, \textit{Anal. Math. Phys.}, {\bf 54}(4)(2022), 1205-1225.
	                                      
\bibitem{Ahamed-Allu-Halder-AMP-2021} {\sc M.B. Ahamed, V. Allu}, and {\sc H. Halder},Bohr radius for certain classes of close-to-convex harmonic mappings, \textit{Anal. Math. Phys.}, {\bf 11}, 111(2021).	 

\bibitem{Alkhaleefah-Kayumov-Ponnusamy-PAMS-2019} {\sc S. A. Alkhaleefah, I. R. Kayumov} and {\sc S. Ponnusamy}, On the Bohr inequality with a 	fixed zero coefficient, {\it Proc. Amer. Math. Soc.} {\bf 147} (2019), 5263--5274.

\bibitem{Ahuja-Anand-jain-MDPI-2020} {\sc O. P. Ahuja, S. Anand}, and {\sc N. K. Jain}, Bohr radius problems for some classes of analytic functions using quantum calculus approach, \textit{Mathematics}, \textbf{8}(2020), 623.

\bibitem{Allu-Arora-JMMA-2022} {\sc V. Allu} and {\sc  V. Arora},Bohr-Rogosinski type inequalities for concave univalent functions, \textit{J. Math. Anal. Appl.} \textbf{520}(2022), 126845. 

\bibitem{Anand-Jain-Kumar-BMMSS-2021} {\sc S. Anand, N. K. Jain}, and {\sc S. Kumar}, Sharp Bohr radius constants for certain analytic functions, \textit{Bull. Malays. Math. Sci. Soc.} (2021) 44:1771-1785.

\bibitem{Aouf-IJMMS-1987} {\sc M. K. Aouf}, On a class of $p$-valent starlike functions of order $\alpha$, \textit{Int. J. Math. Math. Sci.} \textbf{10}(4)(1987), 733-744.


\bibitem{Bhowmik-Das-JMAA-2018} {\sc B. Bhowmik} and {\sc N. Das}, Bohr phenomenon for subordinating families of certain univalent functions, \textit{J. Math. Anal. Appl.} \textbf{462}(2)(2018), 1087-1098. 

\bibitem{Bohr-PLMS-1914} {\sc H. Bohr}, A theorem concerning power series, \textit{Proc. Lond. Math. Soc.} \textbf{2}(13)(1914), 1-5. 

\bibitem{Brannan-GMJ-1969} {\sc D. A. Brannan} and {\sc E. W.  Kirwan}, A covering theorem for typically real functions. \textit{Glasg. Math. J.} \textbf{10}(1969), 153–155 

\bibitem{Duren-1983} {\sc P. L. Duren}, Univalent functions, Springer-Verlag, New York, 1983.

\bibitem{Duren-APNY-1970} {\sc P. L. Duren}, Theory of $ H^p $ Spaces, {\it Academic Press, New York}, 1970.


\bibitem{Gao-Zhou-AMC-2007} {\sc C.Y. Gao} and {\sc S. Q. Zhou}, Certain subclass of starlike functions, \textit{Appl. Math. Comput.} \textbf{187}(1)(2007), 176-182.


\bibitem{Garnett-APNY-1981} {\sc J. B. Garnett}, Bounded Analytic Functions, {\it Academic Press, New York}, (1981).

\bibitem{Goodman-1983} {\sc A. W. Goodman}, Univalent Functions, Vol. I. Mariner Publishing Co., Inc. Tampa (1983). 

\bibitem{Huang-Liu-Ponnusamy-AMP-2020} {\sc Y. Huang, M. S. Liu}, and {\sc S. Ponnusamy}, Refined Bohr-type inequalities with area measure for bounded analytic functions, \textit{Anal.Math.Phys.}, \textbf{10}(4), 50(2020).

\bibitem{Ismagilov-2020-JMAA} {\sc A. Ismagilov, I. R. Kayumov} and {\sc S. Ponnusamy}, Sharp Bohr type inequality, {\it J. Math. Anal. Appl.}  {\bf 489} (2020), 124147.

\bibitem{Janowski-APM-1970} {\sc W. Janowski}, Extremal problems for a family of functions with positive real part and for some related families, \textit{Ann. Polon. Math.} \textbf{23}(1970), 159-177 .

\bibitem{Janowski-APM-1973} {\sc W. Janowski}, Some extremal problems for certain families of analytic functions I, \textit{Ann. Polon. Math.} \textbf{28}(1973), 297-326. 

\bibitem{Kulshrestha-1974} {\sc P. K. Kulshrestha}, Coefficient problem for alpha-convex univalent functions, \textit{Arch. Ration. Mech. Anal.} \textbf{54}(1974), 2050211.

\bibitem{Kayumov-CRACAD-2018} {\sc I. R. Kayumov} and {\sc S. Ponnusamy}, Improved version of Bohr’s inequality, C. R. Acad. Sci. Paris, Ser.I \textbf{356}(2018), 272--277.

\bibitem{Kayumov-Khammatova-JMAA-2021} {\sc I. R. Kayumov, D. M. Khammatova} and {\sc S. Ponnusamy}, Bohr-Rogosinski phenomenon for analytic functions and Ces$ \acute{a} $ro operators,  \textit{J. Math. Anal. Appl.} \textbf{496} (2021), 124824. 

\bibitem{Kayumov-MJM-2022} {\sc I. R. Kayumov, D. M. Khammatova}, and {\sc S. Ponnusamy}, The Bohr Inequality for the Generalized Ces´aro Averaging Operators, \textit{Mediterr. J. Math.} (2022), 19:19, https://doi.org/10.1007/s00009-021-01931-1.


\bibitem{Kayumov-Ponnusamy-Shakirov-Math.Nah.-2018} {\sc I. R. Kayumov, S. Ponnusamy}, and {\sc N. Shakirov}, Bohr radius for locally univalent harmonic mappings, \textit{Math. Nachr.} {\bf 291}(2018), 1757–1768.

\bibitem{Littlewood-PLMS-1925} {\sc J. E. Littlewood}, On inequalities in the theory of functions, {\it Proc. London Math. Soc.} {\bf 23}(1925), 481-519.

\bibitem{Gang Liu-JMAA-2021} {\sc G. Liu}, Bohr-type inequality via proper combination, \textit{J. Math. Anal. Appl.},  503 (2021), 125308.

\bibitem{Liu-Ponnusamy-Wang-RACSAM-2020} {\sc M. S. Liu, S. Ponnusamy} and {\sc J. Wang}, Bohr’s phenomenon for the classes of Quasi-subordination and K-quasiregular harmonic mappings, \textit{RACSAM}, {\bf 114}, 115(2021).

\bibitem{Ma-Minda-1992} {\sc W. C. Ma} and {\sc D. Minda}, A unified treatment of some special classes of univalent functions, In: Proceedings of the Conference on Complex Analysis (Tianjin), 157–169 (1992), Conf. Proc. Lecture Notes
Anal., I, Int. Press, Cambridge.


\bibitem{MacGregor-PAMS-1963} {\sc T. H. MacGregor}, The radius of univalence of certain analytic functions, \textit{Proc. Amer. Math. Soc.} \textbf{14}(1963), 514–520. 

\bibitem{Mocanu-1969} {\sc P. T. Mocanu}, Une propri$\acute{e}$t$\acute{e}$ de convexit$\acute{e}$ généralis$\acute{e}$e dans la th$\acute{e}$orie de la repr$\acute{e}$sentation conforme,
\textit{Mathematica} (Cluj) 34(11)(1969), 127-133. 

\bibitem{Remizova-1963} {\sc M. P. Remizova}, Extremal problems in the class of typically real functions, Izv. Vyss. Ucebn. Zaved. \textit{Mathematika} \textbf{32}(1)(1963), 135-144.

\bibitem{Robertson-AM-1936} {\sc M.I.S. Robertson}, On the theory of univalent functions, \textit{Ann. Math.} \textbf{37}(2)(1936), 374-408.

\bibitem{Rogosinski-MZ-1932} {\sc W. Rogosinski}, \"Uber positive harmonische Entwicklungen und typisch-reelle Potenzreihen, \textit{Math. Z.} \textbf{35}(1)(1932), 93-121. 

\bibitem{Ruscheweyh-PAMS-1975} {\sc S. Ruscheweyh}, New criteria for univalent functions, \textit{Proc. Amer. Math. Soc.} 49(1975), 109-115.

\bibitem{Srivastava-Filomat-2016} {\sc H. M. Srivastava}, {\sc D. R˘aducanu} and {\sc P. Zaprawa}, A certain subclass of analytic functions defined by means of differential subordination, \textit{Filomat} 30(14)(2016), 3743-3757.

\bibitem{Yang-Liu-AAA-2011}  {\sc D. G. Yang} and {\sc J. L. Liu}, A class of analytic functions with missing coefficients, \textit{Abstr. Appl. Anal.}, Art.
ID 456729, 16 pp (2011).

\end{thebibliography}
\end{document}